\documentclass[12pt]{amsart}

\usepackage[margin=1in]{geometry}

\usepackage[utf8]{inputenc}
\usepackage[mathscr]{eucal}
\usepackage{enumerate}

\newtheorem{innercustomthm}{Theorem}
\newenvironment{customthm}[1]
{\renewcommand\theinnercustomthm{#1}\innercustomthm}
{\endinnercustomthm}

\usepackage{amssymb}
\usepackage{amsmath}
\usepackage{amsthm}
\usepackage{enumitem}
\usepackage{amscd}
\usepackage{url}

\usepackage{multirow} 
\usepackage{array}

\usepackage{pgf,tikz}
\usetikzlibrary{arrows.meta,math}
\usetikzlibrary{arrows,matrix}

\numberwithin{equation}{section}
\theoremstyle{plain}
\newtheorem{thm}{Theorem}[section]
\newtheorem{lemma}[thm]{Lemma}
\newtheorem{prop}[thm]{Proposition}
\newtheorem{cor}[thm]{Corollary}

\newtheorem{open}[thm]{Question}
\newtheorem{thmx}{Theorem}

\theoremstyle{definition}
\newtheorem{defn}[thm]{Definition}
\newtheorem{ex}[thm]{Example}

\newtheorem{conject}[thm]{Conjecture}
\newtheorem{remark}[thm]{Remark}

\usepackage{xspace} 

\newcommand{\Z}{\mathbb{Z}}

\newcommand{\Orientations}[1]{\mathbf O(#1)}
\newcommand{\Fourientations}[1]{\mathbf F(#1)}
\newcommand{\SignedCircuits}[1]{\vec{\mathbf C}(#1)}
\newcommand{\SignedCoCircuits}[1]{\vec{\mathbf C^*}(#1)}

\newcommand{\Circuits}[1]{\mathbf C(#1)} 
\newcommand{\CoCircuits}[1]{\mathbf C^*(#1)}

\renewcommand{\vec}{\overrightarrow}

\DeclareMathOperator{\flip}{flip}

\newcommand{\ccequiv}{\mathbf G}           
\newcommand{\ccmin}{\mathbf O^\circ}     
\newcommand{\ori}[1]{\mathcal {#1}}          
\newcommand{\fouri}[1]{\mathcal {#1}}        
\newcommand{\oriat}[1]{\langle #1 \rangle}  
\newcommand{\reverse}[2]{{}_{#1}#2}           
\newcommand{\compat}{\sim} 
\newcommand{\fourmap}[2]{\fouri F (#1,#2)} 

\DeclareMathOperator{\row}{row}
\DeclareMathOperator{\BBY}{\beta}

\DeclareMathOperator{\supp}{supp}

\definecolor{light-gray}{gray}{0.8}
\definecolor{v}{rgb}{0.32,0,0.68}
\definecolor{e}{rgb}{0,0.9,0.3}
\definecolor{r}{rgb}{1,0,0}
\definecolor{lightv}{rgb}{0.4,0,0.7}
\definecolor{darke}{rgb}{0,0.6,0.5}

\title[A consistent Sandpile Torsor Algorithm]{A Consistent Sandpile Torsor Algorithm for Regular Matroids}

\author[C. Ding]{Changxin Ding}
\address{School of Mathematics, Georgia Institute of Technology \\ Atlanta, Georgia 30332-0160, USA}
\email{cding66@gatech.edu}

\author[A. McDonough]{Alex McDonough}
\address{Department of Mathematics, University of Oregon\\
Eugene, Oregon 97403, USA}
\email{alexmcd@uoregon.edu}

\author[L. T\'othm\'er\'esz]{Lilla T\'othm\'er\'esz}
\address{ELTE Eötvös Loránd University, P\'azm\'any P\'eter s\'et\'any 1/C, Budapest, Hungary}
\email{lilla.tothmeresz@ttk.elte.hu}

\author[C.H. Yuen]{Chi Ho Yuen}
\address{Department of Applied Mathematics, National Yang Ming Chiao Tung University, Hsinchu 30009, Taiwan}
\email{chyuen@math.nctu.edu.tw}
\date{\today}


\begin{document}
\begin{abstract}
Every regular matroid is associated with a {\em sandpile group}, which acts simply transitively on the set of bases in various ways.
Ganguly and the second author introduced the notion of {\em consistency} to describe classes of actions that respect deletion-contraction in a precise sense, and proved the consistency of rotor-routing torsors (and uniqueness thereof) for plane graphs.

In this work, we prove that the class of actions introduced by Backman, Baker, and the fourth author, is consistent for regular matroids. More precisely, we prove the consistency of its generalization given by Backman, Santos and the fourth author, and independently by the first author.
This extends the above existence assertion, as well as makes progress on the goal of classifying all consistent actions.
\end{abstract}

\maketitle

\section{Introduction} \hspace {1pt}

For over a century, mathematicians have been interested in enumerative properties of the \emph{spanning trees} of graphs. A remarkable and relatively recent observation is that the set of spanning trees of a graph (and more generally, the \emph{bases} of a \emph{regular matroid}) admit interesting group actions, which bestow on these sets a group-like structure. We are deeply curious about this mysterious algebraic structure, especially in cases where it is surprisingly canonical (in a sense that we will clarify over the next several paragraphs). 

To be more precise, the {\em sandpile group} (also called the critical group, Jacobian, etc.) $S(G)$ of a graph $G$ is a finite abelian group whose size is equal to the number of spanning trees of $G$. The algebraic structure discussed in the previous paragraph is given by a \emph{simply transitive action} of $S(G)$ on the spanning trees of $G$. Loosely speaking, we call such an action a \emph{sandpile torsor action}. To define sandpile torsor actions in a systematic way, it is necessary to work on graphs with some auxiliary data (see~\cite[Theorem 8.1]{Wagner}). 

One possible setup is to work with \emph{embedded} graphs (called \emph{ribbon graphs} or \emph{maps}) equipped with a distinguished \emph{root} vertex. There are at least two known algorithms for associating each rooted embedded graph with a sandpile torsor action: the \emph{rotor-routing algorithm} (see~\cite{Hol08}) and the \emph{Bernardi algorithm} (see~\cite{Bernardi, BW_Bernardi}). These algorithms have been shown to give substantially different actions in general, so this setup does not appear to lead to a ``canonical'' algorithm for assigning sandpile torsor actions~\cite{Ding_torsors,SW_torsors}. 


A related setup which leads to further curiosities is to restrict to \emph{plane graphs} (i.e., \emph{planar} embedded graphs). Remarkably, for plane graphs, the rotor-routing and Bernardi algorithms do not depend on the root vertex (\cite[Theorem 2]{CCG}, \cite[Theorem 5.1]{BW_Bernardi}, for an alternate definition, see \cite{trinity_sandpile}). 
Hence the auxiliary data can be chosen as the planar embedding (without specifying a root vertex).
An algorithm which maps from plane graphs to sandpile torsor actions is called a \emph{sandpile torsor algorithm (on plane graphs)}; hence the planar rotor-routing and Bernardi algorithms are examples of sandpile torsor algorithms on plane graphs. 
Perhaps even more remarkably, the rotor-routing and Bernardi algorithms are equivalent in this context \cite[Theorem 7.1]{BW_Bernardi}. This surprising equivalence lead Klivans to conjecture that this algorithm was in some sense ``canonical'', and all ``nice'' sandpile torsor algorithms on plane graphs must have the same structure \cite[Conjecture~4.7.17]{Klivans}. This conjecture was made precise and proven by Ganguly and the second author \cite{GM}.

The first challenge to resolve this conjecture was to give a suitable definition for a ``nice'' sandpile torsor algorithm. To do this, the authors of \cite{GM} introduce the notion of \emph{consistency}. At a high level, a \emph{consistent} sandpile torsor algorithm is one which ``behaves nicely with respect to deletion and contraction'' (Note that if we delete or contract an edge in a plane graph, there is still a natural embedding for the minor.) In general, sandpile groups do not behave well with respect to deletion and contraction: the additive relation $|S(G)|=|S(G\setminus e)|+|S(G/e)|$ implies that it is almost impossible to relate these groups directly in an algebraically meaningful way. Nevertheless, for specific combinations of group elements and spanning trees, there is a 
natural compatibility condition after taking deletion or contraction.
A consistent sandpile torsor algorithm is essentially one which satisfies the compatibility condition. See~\cite[Definition 4.3]{GM} for the precise (but technical) definition. 
After the definition of consistency was established, the two main results of~\cite{GM} were proving the existence and the uniqueness of a consistent sandpile torsor algorithm on plane graphs.\footnote{More precisely, these is a unique collection of four consistent sandpile torsor algorithms on plane graphs that are all closely related.}

In this paper, we are going to work with {\em regular matroids} from a perspective of {\em oriented matroids}. This approach follows a series of works which use operations on orientations of graphs or regular matroids to study sandpile group actions, with the first systematic work by Backman \cite{Backman17}. Informally, oriented matroids generalize directed graphs: just as each cycle (respectively, minimal cut) of a graph can be oriented in two ways which are the reversal of each other, for each circuit (respectively, cocircuit) $C$ of an oriented matroid, we can label its elements as positive or negative and get a {\em signed circuit (respectively, cocircuit)} in two ways which are the reversal of each other. 
Regular matroids, among many equivalent definitions, are orientable binary matroids, which makes them the largest matroid family that preserves many essential properties of graphical matroids. For example, the family of regular matroids is closed with respect to the matroidal version of deletion and contraction.

For regular matroids, \emph{bases} play the role that spanning trees played for graphs. 
Building off work from Bacher, de la Harpe, and Nagnibeda~\cite{BdlHN}, 
Merino defined the sandpile group of a regular matroid~\cite{Merino}, and showed that its size is also equal to the number of bases of the matroid.
In particular, the regular matroid version of a sandpile torsor algorithm is a map from any regular matroid $M$ (with some auxiliary data) to a free transitive action of the sandpile group of $M$ on the set of bases of $M$. 


Using the auxiliary data of \emph{acyclic circuit-cocircuit signatures}, Backman, Baker, and the fourth author defined a sandpile torsor algorithm for regular matroids which was motivated by polyhedral geometry \cite{BBY}. We will call this sandpile torsor algorithm the \emph{BBY algorithm}\footnote{The (implicit) original name of the corresponding bijections was {\em geometric bijections}, which the fourth author prefers.}. Later, the first author \cite{Ding2} and the fourth author with Backman and Santos \cite{BSY} independently generalized this algorithm to work with the auxiliary data of a larger class of signatures called \emph{triangulating circuit-cocircuit signatures}. 
We will continue to refer to it as the BBY algorithm. It is worthwhile to mention that the aforementioned rotor-routing action for plane graphs also fits into this framework. That is, it is a BBY torsor for some triangulating signature (see Section \ref{sec:planar}).

Deletion and contraction can be defined for circuit-cocircuit signatures as well.
Hence the definition of \emph{consistency} generalizes naturally to the context of sandpile torsor algorithms on regular matroids (equipped with triangulating circuit-cocircuit signatures). 
However, the arguments used in~\cite{GM} to prove the existence and uniqueness of a consistent sandpile torsor algorithm 
frequently reference the vertices of graphs, which do not have a matroidal analogue, hence generalizing these arguments is far from being straightforward. 

Nevertheless, we were able to prove the consistency of the BBY algorithm by applying a framework introduced by the first author~\cite{Ding2}. This framework gives an alternate definition of the BBY algorithm which uses \emph{fourientations}, an object that was first defined by Backman and Hopkins~\cite{BH}. Our proof essentially comes down to classifying ways that consistency could be violated and then showing that each of these potential possibilities leads to a contradiction. 
Our main result proves \cite[Conjecture 6.11]{GM} along with a variant of it (see Corollaries~\ref{cor:BBYacyclic} and~\ref{cor:BBYconsistent}). We leave the question of uniqueness as a conjecture (see Conjecture~\ref{conj:uniqueness} which is a variant of \cite[Conjecture 6.14]{GM}).



\subsection{Statement of the Main Theorem}

In this subsection, we give a high-level introduction of the main notions, then state our main theorem. The precise definitions can be found in later sections.


For an oriented regular matroid $M$ on ground set $E$, the sandpile group is defined as \[S(M) := \frac{\mathbb Z^E}{\Lambda(M) \oplus \Lambda^*(M)},\] where $\Lambda(M) \subseteq \mathbb{Z}^E$ is the lattice of flows, and $\Lambda^*(M) \subseteq \mathbb{Z}^E$ is the lattice of dual flows. 
Note that the sandpile group is generated by (classes of) arcs, which are oriented elements of $M$ interpreted as the standard basis elements of $\mathbb{Z}^E$ and their negations.

As described earlier, in the matroidal setting, we use a circuit-cocircuit signature as an auxiliary structure to define a sandpile torsor (the BBY torsor). A circuit-cocircuit signature is a collection of signed circuits and signed cocircuits such that for each circuit and cocircuit, we choose exactly one of the two signed circuits/cocircuits supported on the circuit/cocircuit. 
The aforementioned triangulating circuit-cocircuit signatures are signatures that satisfy certain ``non-overlapping'' condition (see Definition \ref{def:triangulating} for the rigorous definition). The family of triangulating circuit-cocircuit signatures are closed under deletion and contraction of signatures.


A triangulating circuit-cocircuit signature $(\sigma,\sigma^*)$ of $M$ first gives rise to a bijection (BBY bijection) $\BBY_{(M,\sigma,\sigma^*)}$ between bases of the matroid and intermediate objects known as {\em circuit-cocircuit equivalence classes} of orientations. The sandpile group of $M$ acts on these equivalence classes in a natural way. The BBY torsor is defined as the composition of this action and the BBY bijection: the result of $s\in S(M)$ acting on a basis $B_1$ is the unique basis $B_2$ such that $s\cdot\BBY_{(M,\sigma,\sigma^*)}(B_1)=\BBY_{(M,\sigma,\sigma^*)}(B_2)$, where $\cdot$ is the action of $S(M)$ on the circuit-cocircuit equivalence classes.

The consistency of the BBY torsor means that this action, defined using the BBY bijection and canonical action, behaves nicely with respect to deletion and contraction of both the matroid and the triangulating signature: the action of an arc on a basis can be descended to the deletion or contraction as long as it is well-defined.
We also require that the action of an arc only modifies a basis in the arc's connected component. More formally, this is the consistency of the BBY torsor for regular matroids (our main theorem):

\begin{thmx}\label{thm:consistency}
    Let $M$ be a regular matroid and $(\sigma,\sigma^*)$ be a triangulating circuit-cocircuit signature. Suppose that $\vec f$ is an arc and $B_1,B_2 \in \mathbf B(M)$ such that
    \[[\vec f] \cdot \BBY_{(M,\sigma,\sigma^*)}(B_1)=\BBY_{(M,\sigma,\sigma^*)}(B_2).\] Then, the following 3 properties must hold. 
\begin{enumerate}
    \item For any $e \in (B_1^c \cap B_2^c)\setminus f$, we have
    \[ [\vec f] \cdot \BBY_{(M\setminus e,\sigma\setminus e,\sigma^*\setminus e)}(B_1)=\BBY_{(M\setminus e,\sigma\setminus e,\sigma^*\setminus e)}(B_2).\]
    \item For any $e \in (B_1 \cap B_2)\setminus f$, we have
    \[ [\vec f] \cdot \BBY_{(M/ e,\sigma/ e,\sigma^*/ e)}(B_1\setminus e)=\BBY_{(M/ e,\sigma/ e,\sigma^*/ e)}(B_2\setminus e).\]
    \item For any $e\in E$ that is in a different connected component of $M$ than $f$, we have 
    \[e \in B_1 \iff e \in B_2.\]
\end{enumerate}
\end{thmx}

\subsection{Organization and Convention of the Paper}
In Section~\ref{sec:background}, we provide the necessary background to state our main theorem (Theorem~\ref{thm:consistency}). In Section~\ref{sec:torsors}, we describe the general notion of sandpile torsor algorithms and their consistency, and point out that it unifies several settings in the literature, e.g. acyclic circuit-cocircuit signatures and planar embeddings. Then, we specialize Theorem~\ref{thm:consistency} to deduce the consistency of the sandpile torsor algorithms in these instances. In Section~\ref{sec:action}, we prove a technical theorem which gives a complete description of the \emph{canonical action} of arcs on \emph{circuit-cocircuit minimal orientations} (Theorem~\ref{thm:generalDescription}). This theorem is then applied in Section~\ref{sec:main} to prove Theorem~\ref{thm:consistency}. 
Finally, in Section~\ref{sec:further}, we discuss the uniqueness conjecture and a possible strategy to attack it. 

In this paper, we use a linear algebraic approach to describe the oriented matroid structure of a regular matroid. In particular, we fix a totally unimodular realization $A$ of the matroid and view orientations and signed circuits/cocircuits as elements in $\mathbb{Z}^E$.
This approach does not lose any generality due to the following rigidity result.

\begin{thm} \cite[Corollary~7.9.4]{Oriented}
Let $M$ be a regular matroid.
Then all oriented matroid structures of $M$ can be obtained from reorienting an arbitrary totally unimodular (or more generally, real) matrix $A$ representing the underlying matroid of $M$.
\end{thm}

Recall that reorienting an oriented matroid in the context of realization means negating some columns of the matrix, or equivalently, the signed circuits and cocircuits of the new oriented matroid differ from the original one by a universal sign flip on the respective coordinates.
It is easy to see that reorientations do not change the combinatorics we consider in this paper.

\subsection{Acknowledgement}

AM was partially supported by NSF grant DMS-2039316. LT was supported by the National Research, Development and Innovation Office of Hungary -- NKFIH, grant no.\ 132488, by the János Bolyai Research Scholarship of the Hungarian Academy of Sciences, and by the \'UNKP-23-5 New National Excellence Program of the Ministry for Innovation and Technology, Hungary. LT was also partially supported by the Counting in Sparse Graphs Lendület Research Group of the Alfr\'ed Rényi Institute of Mathematics.
CHY was supported by the Trond Mohn Foundation project ``Algebraic and Topological Cycles in Complex and Tropical Geometries'', the Danish National Research Foundation project DNRF151, and partially supported by the Ministry of Science and Technology of Taiwan project MOST 113-2115-M-A49-004-MY2 during his affiliation to the University of Oslo, the University of Copenhagen, and National Yang Ming Chiao Tung University, respectively.
All four authors thank Spencer Backman, Donggyu Kim, and the anonymous referees for their helpful comments.

\section{Background and Notation} \label{sec:background}


We call elements of $\Z^{E}$ \emph{(integral) 1-chain}s. 
For a 1-chain $\vec P$, and some $e \in E$, we write $\vec P \oriat{e}$ for the coefficient of $e$ in $\vec P$. We also write $P:= \{e \in E : \vec P \oriat{e} \not= 0\}$ for the \emph{support} of $\vec P$.
For $e\in E$, denote by $\vec P\setminus e$ the 1-chain in $\mathbb{Z}^{E\setminus e}$ obtained by restricting $\vec P$ to $E\setminus e$. 

An integral 1-chain is \emph{simple} if every coefficient is in $\{-1,0,1\}$. 
An \emph{arc} is a simple 1-chain whose support has only one element. We will write arcs in the form $\vec e$, where $e\in E$.

\subsection{Oriented Matroids and Regular Matroids}\label{sec:matroids} 

\hspace{1 pt} 

Let $A$ be an $r\times m$ {\em totally unimodular} matrix of full row rank, i.e., a matrix over the reals in which the determinant of every square submatrix is either $-1,1$, or $0$. $A$ will be fixed for the rest of this paper.
Let $E$ be a set that indexes the columns of $A$.

\begin{defn}
The {\em regular matroid} $M:=M(A)$ represented by $A$ consists of the set $E$ and the collection of subsets $\mathbf B(M)\subseteq\binom{E}{r}$ corresponding to the nonzero maximal minors of $A$. The elements of $\mathbf B(M)$ are called the \emph{bases} of $M$. 

A subset $C\subseteq E$ is a {\em circuit} of $M$ if it is an inclusionwise minimal subset that is not contained in any basis; a subset $C^*\subseteq E$ is a {\em cocircuit} of $M$ if it is an inclusionwise minimal subset that intersects every basis. We denote the set of circuits of $M$ by $\Circuits M$ and the set of cocircuits of $M$ by $\CoCircuits M$.
%
\end{defn}

The following lemma is well known, but we provide a proof for convenience. 
\begin{lemma}\cite[Lemma 8]{su2010lattice}\label{lem:ker row} 
For every $C \in \Circuits M$ (respectively, $C^* \in \CoCircuits M$), there exist exactly two simple 1-chains in $\ker(A)$ (respectively, $\row(A)$) whose support is $C$ (respectively, $C^*$).
Moreover, these two elements are the negation of each other.
\end{lemma}
\begin{proof}
It follows from Cramer's rule, and the fact that $A$ is totally unimodular, that for any circuit $C$, there exists an element in $\ker(A)$ with support $C$ whose coordinates are in $\{-1,0,1\}$. The minimality in the definition of a circuit ensures that for those elements of $\ker(A)$ whose support is $C$, the coordinates are unique up to scalar multiplication. Hence, there are only two such elements with coordinates in $\{-1,0,1\}$, and they are the negations of each other. The argument is similar when working with cocircuits. 
\end{proof}

We call a simple 1-chain in $\ker(A)$ (resp. $\row(A)$) whose support is a circuit (resp. cocircuit) a \emph{signed circuit} (resp. \emph{signed cocircuit}) of $M$.  
The collection of signed circuits and signed cocircuits of $M$ are denoted by $\vec{\mathbf C}(M)$ and $\vec{\mathbf C^*}(M)$ respectively.

The following result gives a useful property of 1-chains in $\ker(A)$ or $\row(A)$ that is specific to regular matroids (or equivalently, when $A$ is totally unimodular). 

\begin{lemma}\cite[Lemma~4.1.1.]{BBY}\label{lem:decompose}
Let $A$ be a totally unimodular matrix, and $\vec P$ be a $1$-chain. 
\begin{enumerate}
    \item If $\vec P\in\ker(A)$, then $\vec P$ can be written as the sum of a collection of signed circuits $\vec C$ such that $\vec P \oriat{e}\cdot \vec C \oriat{e}>0$ for any $e\in C$.
    In particular, if $\vec P$ is a simple $1$-chain in $\ker(A)$, then $\vec P$ can be written as a sum of signed circuits of disjoint support. 
    \item If $\vec P\in\row(A)$, then $\vec P$ can be written as the sum of a collection of signed cocircuits $\vec{C^*}$ such that $\vec P \oriat{e}\cdot \vec{C^*} \oriat{e}>0$ for any $e\in C^*$.
    In particular, if $\vec P$ is a simple $1$-chain in $\row(A)$, then $\vec P$ can be written as a sum of signed cocircuits of disjoint support.
\end{enumerate}
\end{lemma}



The orthogonality between the kernel and row space of a matrix, hence between signed circuits and signed cocircuits, implies the following fact immediately.

\begin{lemma}\label{lem:ccint} Suppose $\vec C$ is a signed circuit and $\vec {C^*}$ is a signed cocircuit of a regular matroid. The number of $x \in C \cap C^*$ such that $\vec C\oriat{x} = \vec {C^*} \oriat{x}$ is equal to the number of $y \in C \cap C^*$ such that $\vec C\oriat{y} = -\vec {C^*} \oriat{y}$
\end{lemma}

Another fundamental (and useful) notion is the dual of an oriented matroid.

\begin{defn} \label{def:dual_OM}
The \textit{dual} of an oriented matroid $M$, denoted by $M^*$, is an oriented matroid on the same ground set $E$ whose bases are $\{E\setminus B: B\in\mathbf{B}(M)\}$, and that $\SignedCircuits{M^*}=\SignedCoCircuits{M}$ and $\SignedCoCircuits{M^*}=\SignedCircuits{M}$.
\end{defn}
It is a standard result (e.g., \cite[Section~7.9]{Oriented}) that the dual of an oriented regular matroid exists, and it is also regular.

A matroid is said to be \emph{connected} if for any two elements of the ground set, there is a circuit containing both of them. The property of having a circuit containing two given elements is an equivalence relation on the elements, and its classes are called the \emph{connected components} of the matroid.



An \emph{orientation} is a map from $E$ to $\{-,+\}$. Adopting the usual convention, we write $\ori{O} \oriat{x}$ for the value of the orientation at $x$. Note that we do not think of orientations as integral 1-chains, even though the concepts are related. We denote by $\Orientations{M}$ the set of all orientations of $M$. 

\begin{defn}\label{def:reverse} 
Let $\ori O$ be an orientation and $P$ be a subset of $E$. Then we write $\reverse{P}{\ori{O}}$ for the orientation obtained by switching the sign of $\ori O$ at each element of $P$. In other words, for $x \in E$, we have
\[ \reverse{P} {\ori{O}}\oriat{x}  = \begin{cases}- \ori{O}\oriat{x}  & \text{if $x \in P$,}\\  \ori{O}\oriat{x}  & \text{if $x \not\in P$.}\end{cases}\]
\end{defn}

\begin{defn}\label{def:oricompat}
    Let $\ori O$ be an orientation and $\vec P$ be a 1-chain. We say that $\vec P$ is \emph{compatible with $\ori O$} if for all $f \in P$, the sign of $\vec P\oriat{f}$ matches the sign of $\ori O \oriat{f}$. We denote compatibility by writing $ \vec P\compat \ori O$.
\end{defn}

Let $\ori O \in \Orientations{M}$ and $\vec C$ be a signed circuit that is compatible with $\ori O$. We say that $\reverse{C} {\ori O}$ is a \emph{circuit reversal} of $\ori O$. 
Define \emph{cocircuit reversals} analogously.
Two orientations $\ori {O}_1$ and $\ori {O}_2$ differ by \emph{circuit-cocircuit reversals} if $\ori {O}_1$ can be sent to $\ori {O}_2$ by a sequence of circuit and/or cocircuit reversals. It is easy to show that this is an equivalence relation on $\Orientations{M}$. 


More generally, the definition of circuit and cocircuit reversals can also be applied to any simple $1$-chain $\vec P$: If $\vec P \compat \ori O$, then $\reverse{P}{\ori{O}}$ is a \emph{reversal} of $\vec P$. Note that Definition~\ref{def:reverse} only concerns 
a subset of $E$, while here, we talk about a 1-chain. 
\begin{ex}
    Take the graphic matroid on Figure \ref{fig:Changxin's_graph}, and take the orientation $\ori O$ with $\ori O\oriat{f_1} = +$, $\ori O\oriat{f_2} = -$, $\ori O\oriat{f_3} = +$ and $\ori O\oriat{f_4} = +$ (in short, $(+,-,+,+)$; we will use this shorthand throughout). The signed circuit $\vec C=\vec {f_1} - \vec {f_2} + \vec {f_3}$ is compatible with $\ori O$. By reversing $\vec C$, we get the orientation $(-,+,-,+)$.
\end{ex}

\begin{defn}The \emph{circuit-cocircuit equivalence classes} of $M$ are the orientations of $M$ modulo the equivalence relation defined in the previous paragraph. The set of these equivalence classes is denoted $\ccequiv(M)$. For any element $\ori O \in \Orientations M$, we write $[\ori O]$ for the equivalence class of $\ccequiv(M)$ containing $\ori O$.
\end{defn}

The set $\ccequiv(M)$ was first explored by Gioan~\cite{Gioan, Gioan2}, and it serves as an intermediate object to define the BBY action because of the following enumerative fact. 

\begin{figure}
    \begin{center}
    \begin{tikzpicture}[scale=0.8]
    \tikzstyle{o}=[circle,fill,scale=.5,draw]
	\begin{scope}[shift={(-4,0)}]
		\node [o] (1) at (0,0) {};
	\node [o] (2) at (1,1.8) {};
	\node [o] (3) at (2,0) {};
	\draw [thick,->,>=stealth'] (2) to node[fill = white,inner sep=1pt]{\footnotesize $f_1$} (1);
	\draw [thick,->,>=stealth',bend left=30] (1) to node[fill = white,inner sep=1pt]{\footnotesize $f_3$} (3);
	\draw [thick,->,>=stealth',bend right=30] (1) to node[fill = white,inner sep=1pt]{\footnotesize $f_4$} (3);
	\draw [thick,->,>=stealth'] (2) to node[fill = white,inner sep=1pt]{\footnotesize $f_2$} (3);
 
	\end{scope}

    \begin{scope}[shift={(4, 0.8)}]
	\node at (0,0) {$\begin{bmatrix}
	        1 & 0 & -1 & -1 \\
            -1& -1&  0 & 0  \\
            0 & 1 &  1 & 1
	\end{bmatrix}$};
    \end{scope}
\end{tikzpicture}
\end{center}
\caption{A graph (graphic matroid) and its corresponding representing matrix.}\label{fig:Changxin's_graph}
\end{figure}
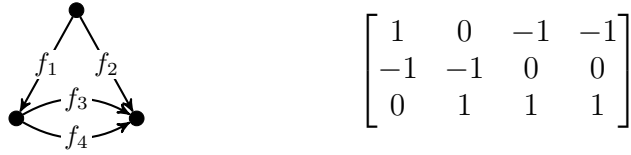
\begin{thm} \cite{Gioan2}
For a regular matroid $M$, we have $|\ccequiv(M)|=|\mathbf{B}(M)|$.
\end{thm}
Moreover, the sandpile group has a natural canonical torsor structure on $\ccequiv(M)$ (described in the next section). To define a bijection between $\ccequiv(M)$ and $\mathbf{B}(M)$ one needs some auxiliary structure. The BBY bijection (described in Section \ref{sec:BBY_bij}) is one such bijection using a circuit-cocircuit signature.

The following lemma is useful for exploring the set $\ccequiv(M)$. 
\begin{lemma}\cite[Corollary 3.4.6]{Oriented}\label{lem:circorcocirc}
    Consider $\ori O \in \Orientations M$ and $x \in E$. Precisely one of the following is true.
    \begin{enumerate}
        \item These exists some $\vec C \in \SignedCircuits{M}$ such that $x \in C$ and $\vec C \compat \ori O$.
        \item These exists some $\vec {C^*} \in \SignedCoCircuits{M}$ such that $x \in C^*$ and $\vec {C^*} \compat \ori O$.
    \end{enumerate}
\end{lemma}

\begin{cor}\label{cor:arccompat}
    Consider an arc $\vec f$ and an orientation $\ori O \in \Orientations M$. There exists some orientation $\ori O'\in \Orientations M$ such that $\vec{f} \compat \ori O'$ and $[\ori O'] = [\ori O]$. 
\end{cor}
\begin{proof}
    If $\vec{f} \compat \ori O$, then we simply let $\ori O' = \ori O$. Otherwise, we apply Lemma~\ref{lem:circorcocirc} with respect to the element $f$ and set $\ori O' = \reverse{C}\ori O$ or $\ori O' = \reverse{C^*} \ori O$ depending on which case of the lemma applies.
\end{proof}

\subsection{The sandpile group and its canonical action on $\ccequiv(M)$}\hspace{1 pt} \label{sec:BBY_act}

\begin{defn}Let $\Lambda(M) \subseteq \mathbb{Z}^E$ be the lattice generated $\SignedCircuits{M}$ and $\Lambda^*(M) \subseteq \mathbb{Z}^E$ be the lattice generated by $\SignedCoCircuits{M}$. The \emph{sandpile group} of $M$ is defined by:
\[ S(M) := \frac{\mathbb Z^E}{\Lambda(M) \oplus \Lambda^*(M)}.\]
\end{defn}

The sandpile group also has many other names such as the \emph{Jacobian group} or \emph{critical group}. For a 1-chain $\vec{P}$, we write $[\vec P]$ for the equivalence class of $S(M)$ containing $\vec{P}$. Note that the sandpile group $S(M)$ is generated by elements $\{[\vec f]\mid  \vec f \text{ is an arc of }M\}$.
In \cite{BBY}, the authors define a natural group action of $S(M)$ on the set $\ccequiv(M)$, which generalizes the additive action in the more classical graphical case where elements of $S(M)$ and $\ccequiv(M)$ are represented as ``chip configurations'' \cite{Backman17}. For details on the ``chip'' perspective, see~\cite{Klivans}.

\begin{defn}\cite{BBY}\label{def:groupaction}
    The \emph{canonical action} of $S(M)$ on $\ccequiv(M)$ is defined by linearly extending the following action of each generator $[\vec f]$ on circuit-cocircuit reversal classes. 
    \begin{enumerate}
    \item Suppose we are given an arc $\vec f$ and some $\ori G \in \ccequiv(M)$.
    \item By Corollary~\ref{cor:arccompat}, there must exist some $\ori O\in \ori G$ such that $-\vec f \compat \ori O$.
    \item Define the action by $[\vec f]\cdot \ori G=[ \reverse{f} \ori O]$.
    \end{enumerate}
\end{defn}


\begin{thm}\cite[Theorem 4.3.1.]{BBY}\label{thm:simplytransitive} The canonical action is well-defined and simply transitive. 
\end{thm} 

\begin{ex}\label{ex:canonical_action_without_representatives}
    Take the graphic matroid on Figure \ref{fig:Changxin's_graph}, and take the orientation $\ori O$ with $\ori O\oriat{f_1} = -$, $\ori O\oriat{f_2} = -$, $\ori O\oriat{f_3} = +$, and $\ori O\oriat{f_4} = +$ (in short, $\ori O =(-,-,+,+)$). Since $-\vec f_1 \sim \ori O$, we have $[\vec f_1]\cdot [\ori O] = [(+,-,+,+)]$.

    As a more interesting example, consider also $[\vec {f_3}]\cdot [\ori O]$. Since $\vec {f_3}\sim \ori O$, we need to reverse a signed circuit or cocircuit containing $f_3$, and then reverse $f_3$ again. Notice that $-\vec {f_1}+\vec {f_3} + \vec {f_4}$ is a signed cocircuit containing $f_3$ that is compatible with $\ori O$. Hence $[(+,-,+,-)]$ is the circuit-cocircuit equivalence class of $[\vec {f_3}]\cdot [\ori O]$.
\end{ex}

As such, any bijection between $\ccequiv(M)$ and $\mathbf{B}(M)$ yields a simply transitive group action of $S(M)$ on $\mathbf{B}(M)$ via composing with the canonical action.

\subsection{Fourientations} \label{sec:four}
We next introduce a bijection between $\ccequiv(M)$ and $\mathbf{B}(M)$.
It is convenient to introduce our theory in terms of a generalization of orientations called \emph{fourientations}. These objects were systematically studied by Backman and Hopkins \cite{BH}, but we will only make use of the basic notions.

\begin{defn}
    Given a set $E$, a \emph{fourientation} $\fouri{F}$ is a map from $E$ to the set $\{\emptyset,-,+,\pm\}$. Here, we use $-$ for the set $\{-\}$, we use $+$ for the set $\{+\}$, and we use $\pm$ for the set $\{-,+\}$. We denote the set of fourientations on the ground set of a matroid $M$ by $\Fourientations{M}$. 
\end{defn}

As with orientations, for $x \in E$, we write $\fouri{F} \oriat{x}$ for the output of the map $\fouri{F}$ at $x$. Intuitively, each element of the ground set can be oriented in either direction, bi-oriented, or unoriented. 

Given a fourientation $\fouri F$, we define two fourientations $-\fouri F$ and $\fouri F^c$ by: 
\[ -\fouri{F}\oriat{e} = \begin{cases}- & \text{ if } \fouri F\oriat{e} = +\\ + & \text{ if } \fouri F\oriat{e} = -\\ \pm & \text{ if } \fouri F\oriat{e} = \pm\\ \emptyset & \text{ if } \fouri F\oriat{e} = \emptyset\end{cases} \hspace{1 cm} \text{ and }\hspace{1 cm}
 \fouri{F}^c\oriat{e} = \begin{cases}- & \text{ if } \fouri F\oriat{e} = +\\ + & \text{ if } \fouri F\oriat{e} = -\\ \emptyset & \text{ if } \fouri F\oriat{e} = \pm\\ \pm & \text{ if } \fouri F\oriat{e} = \emptyset\end{cases}\]


Below, we give a natural generalization of Definition~\ref{def:oricompat} to the fourientation context. 
\begin{defn}\label{def:fouricompat}
 Let $\fouri F$ be a fourientation, $\vec P$ be a $1$-chain, and $x \in E$. We write $\vec P \oriat{x} \compat \fouri F\oriat{x}$ if any of the following conditions hold:
 \renewcommand{\theenumi}{\roman{enumi}}
 \begin{enumerate}
    \item $\vec P\oriat{x} = 0$,
    \item $\fouri F \oriat{x} = \pm$,
    \item $\vec P\oriat{x} <0$ and $\fouri F \oriat{x} = -$, or 
    \item $\vec P\oriat{x} >0$ and $\fouri F \oriat{x} = +$.
\end{enumerate}

We say that $\vec P$ is \emph{compatible} with $\fouri F$, denoted $\vec P \compat \fouri F$, if $\vec P\oriat{x} \compat \fouri F\oriat{x}$ for every $x \in E$. 
\end{defn}

An equivalent definition is that $\vec P$ is compatible with $\fouri F$  if for every $x \in P$, the sign of $\vec P\oriat{x}$ is contained in $\fouri F\oriat{x}$. Note that Definition~\ref{def:fouricompat} is equivalent to Definition~\ref{def:oricompat} when restricted to orientations, so the common notation for compatibility does not introduce any ambiguity. 


Like compatibility, reversing subsets of $E$ also makes sense in the context of fourientations. In particular, we have the following generalization of Definition~\ref{def:reverse}. 
\begin{defn}\label{def:Freverse} Let $\fouri F\in \Fourientations{M}$ and $P$ be a subset of $E$. Define $\reverse{P}{\fouri F}$ to be the fourientation such that, for $x \in E$, we have
\[ \reverse{P} {\ori{F}}\oriat{x}  = \begin{cases}- \ori{F}\oriat{x}  & \text{if $x \in P$ and $\ori{F}\oriat{x}\in\{-,+\}$,}\\  \ori{F}\oriat{x}  & \text{otherwise.}\end{cases}\]
If $\vec P$ is a 1-chain and $\vec P \compat \fouri F$, then we say that $\reverse {P}{\fouri F}$ is a \emph{reversal} of $\fouri F$ by $\vec P$.
\end{defn}


The following lemma is called the {\em 3-painting axiom}, which generalizes Lemma~\ref{lem:circorcocirc}. It will be crucial for many of our future arguments. 
\begin{lemma}\cite[Theorem 3.4.4]{Oriented}\label{lem:3painting}
    Let $\fouri F$ be a fourientation and $x \in E$ such that $\fouri F\oriat{x} \in \{-,+\}$. Then, exactly one of the following conditions holds:
    
    \begin{enumerate}
        \item There exists a signed circuit $\vec C$ that is compatible with $\fouri F$ such that $x\in C$.
        \item There exists a signed cocircuit $\vec {C^*}$ that is compatible with $(-\fouri F)^c$ such that $x\in C^*$.
    \end{enumerate}
\end{lemma} 

\subsection{Triangulating Signatures and the Backman--Baker--Yuen Bijection} \label{sec:BBY_bij}

In \cite{BBY}, Backman, Baker, and the fourth author defined a family of explicit bijections between $\ccequiv(M)$ and $\mathbf{B}(M)$. These maps were generalized in \cite{BSY} and \cite{Ding2} to the context we will use in this paper. Below, we give their constructions in the language of fourientations. 

The key ingredient to define the bijection is a circuit-cocircuit signature $(\sigma, \sigma^*)$, which we define below. 

\begin{defn}\label{def:signatures} 
A \emph{circuit signature} $\sigma \subset \SignedCircuits M$ is a collection of signed circuits of $M$ such that for each circuit $C \in \Circuits{M}$, exactly one of the two signed circuits supported on $C$ is contained in $\sigma$. We write $\sigma(C)$ for the signed circuit supported on $C$ that is contained in $\sigma$. 

A \emph{cocircuit signature} $\sigma^* \subset \SignedCoCircuits M$ is defined analogously. For a cocircuit $C^*$, we write $\sigma^*(C^*)$ for the signed cocircuit supported on $C^*$ that is contained in $\sigma^*$. 

By a \emph{circuit-cocircuit signature} we  mean a pair consisting of a circuit signature and a cocircuit signature.
\end{defn}

\begin{ex}\label{ex:signature}
For the graph of Figure \ref{fig:Changxin's_graph}, the signed circuits $\vec {f_1}-\vec {f_2} +\vec {f_3}$, $\vec {f_1} - \vec {f_2} + \vec {f_4}$ and $-\vec {f_3} + \vec {f_4}$ form a circuit signature. The signed cocircuits $-\vec {f_1} + \vec {f_3} + \vec {f_4}$, $-\vec {f_1} -\vec {f_2}$, and $\vec {f_2} + \vec {f_3} + \vec {f_4}$ form a cocircuit signature. See also Figure \ref{fig:signature}.
\end{ex}

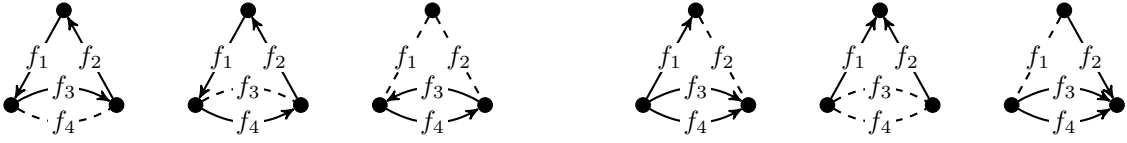
\begin{figure}
	\begin{center}
		\begin{tikzpicture}[scale=0.7]
			\tikzstyle{o}=[circle,fill,scale=.5,draw]
			\begin{scope}[shift={(-9.5,0)}]
				\node [o] (1) at (0,0) {};
				\node [o] (2) at (1,1.8) {};
				\node [o] (3) at (2,0) {};
				\draw [thick,->,>=stealth',] (2) to node[fill = white,inner sep=1pt]{\footnotesize $f_1$} (1);
                \draw [thick,<-,>=stealth',] (2) to node[fill = white,inner sep=1pt]{\footnotesize $f_2$} (3);
				\draw [thick,->,>=stealth',bend left=30] (1) to node[fill = white,inner sep=1pt]{\footnotesize $f_3$} (3);
				\draw [thick,-,>=stealth',dashed,bend right=30] (1) to node[fill = white,inner sep=1pt]{\footnotesize $f_4$} (3);
			\end{scope}
			
			\begin{scope}[shift={(-6,0)}]
				\node [o] (1) at (0,0) {};
				\node [o] (2) at (1,1.8) {};
				\node [o] (3) at (2,0) {};
				\draw [thick,->,>=stealth',] (2) to node[fill = white,inner sep=1pt]{\footnotesize $f_1$} (1);
                \draw [thick,<-,>=stealth',] (2) to node[fill = white,inner sep=1pt]{\footnotesize $f_2$} (3);
				\draw [thick,-,>=stealth',dashed,bend left=30] (1) to node[fill = white,inner sep=1pt]{\footnotesize $f_3$} (3);
				\draw [thick,->,>=stealth',bend right=30] (1) to node[fill = white,inner sep=1pt]{\footnotesize $f_4$} (3);
			\end{scope}
			
			\begin{scope}[shift={(-2.5,0)}]
				\node [o] (1) at (0,0) {};
				\node [o] (2) at (1,1.8) {};
				\node [o] (3) at (2,0) {};
				\draw [thick,-,>=stealth',dashed] (2) to node[fill = white,inner sep=1pt]{\footnotesize $f_1$} (1);
                \draw [thick,-,>=stealth',dashed] (2) to node[fill = white,inner sep=1pt]{\footnotesize $f_2$} (3);
				\draw [thick,<-,>=stealth',bend left=30] (1) to node[fill = white,inner sep=1pt]{\footnotesize $f_3$} (3);
				\draw [thick,->,>=stealth',bend right=30] (1) to node[fill = white,inner sep=1pt]{\footnotesize $f_4$} (3);
			\end{scope}

			\begin{scope}[shift={(2.5,0)}]
				\node [o] (1) at (0,0) {};
				\node [o] (2) at (1,1.8) {};
				\node [o] (3) at (2,0) {};
				\draw [thick,<-,>=stealth'] (2) to node[fill = white,inner sep=1pt]{\footnotesize $f_1$} (1);
                \draw [thick,-,>=stealth',dashed] (2) to node[fill = white,inner sep=1pt]{\footnotesize $f_2$} (3);
				\draw [thick,->,>=stealth',bend left=30] (1) to node[fill = white,inner sep=1pt]{\footnotesize $f_3$} (3);
				\draw [thick,->,>=stealth',bend right=30] (1) to node[fill = white,inner sep=1pt]{\footnotesize $f_4$}(3);
			\end{scope}
			
			\begin{scope}[shift={(6,0)}]
				\node [o] (1) at (0,0) {};
				\node [o] (2) at (1,1.8) {};
				\node [o] (3) at (2,0) {};
				\draw [thick,<-,>=stealth'] (2) to node[fill = white,inner sep=1pt]{\footnotesize $f_1$} (1);
                \draw [thick,<-,>=stealth'] (2) to node[fill = white,inner sep=1pt]{\footnotesize $f_2$} (3);
				\draw [thick,-,>=stealth',bend left=30,dashed] (1) to node[fill = white,inner sep=1pt]{\footnotesize $f_3$} (3);
				\draw [thick,-,>=stealth',bend right=30,dashed] (1) to node[fill = white,inner sep=1pt]{\footnotesize $f_4$} (3);
			\end{scope}
			
			\begin{scope}[shift={(9.5,0)}]
				\node [o] (1) at (0,0) {};
				\node [o] (2) at (1,1.8) {};
				\node [o] (3) at (2,0) {};
				\draw [thick,-,>=stealth',dashed] (2) to node[fill = white,inner sep=1pt]{\footnotesize $f_1$} (1);
                \draw [thick,->,>=stealth'] (2) to node[fill = white,inner sep=1pt]{\footnotesize $f_2$} (3);
				\draw [thick,->,>=stealth',bend left=30] (1) to node[fill = white,inner sep=1pt]{\footnotesize $f_3$} (3);
				\draw [thick,->,>=stealth',bend right=30] (1) to node[fill = white,inner sep=1pt]{\footnotesize $f_4$} (3);
			\end{scope}
		\end{tikzpicture}
	\end{center}
	\caption{The circuit signature of Example \ref{ex:signature} (left panels) and the cocircuit signature of Example \ref{ex:signature} (right panels). 
	}
	\label{fig:signature}
\end{figure}


Fix a circuit-cocircuit signature $(\sigma,\sigma^*)$ and a basis $B$. For each $e\not\in B$, let $C_e$ be the unique circuit contained in $B\cup\{e\}$ (known as the {\em fundamental circuit} of $e$ with respect to $B$). For each $e \in B$, let $C^*_{e}$ be the unique cocircuit contained in $(E \setminus B)\cup\{e\}$ (known as the {\em fundamental cocircuit} of $e$ with respect to $B$). 

\begin{defn}\label{def:basisfouri} 
We denote by $\fourmap{B}{\sigma}$ the fourientation where all $e \in B$ are bi-oriented and all the $e \in E \setminus B$ are oriented according to $\sigma(C_e)$.
Similarly, we denote by $\fourmap{B}{\sigma^*}$ the fourientation where all $e \in E \setminus B$ are bi-oriented and all $e \in B$ are oriented according to $\sigma^*(C^*_e)$.
\end{defn}
\begin{ex}\label{ex:fourientation}
For the graph of Figure \ref{fig:Changxin's_graph}, take the spanning tree $T$ consisting of edges $f_1$ and $f_3$. Take the circuit-cocircuit signature $(\sigma,\sigma^*)$ of Example \ref{ex:signature}. Then 
$\fourmap{T}{\sigma}=(\pm,-,\pm,+)$ and 
$\fourmap{T}{\sigma^*}=(-,\pm,+,\pm)$.
\end{ex}

For $\fouri{F}_1, \fouri{F}_2\in \Fourientations{M}$, we write $\fouri{F}_1 \cup \fouri{F}_2$ and $\fouri{F}_1 \cap \fouri{F}_2$ for the fourientations obtained by taking pointwise union or intersection.


\begin{defn}\cite{Ding2}\label{def:triangulating} A circuit $\sigma$ (resp. cocircuit signature $\sigma^*$) is called \emph{triangulating} if for any distinct $B_1,B_2 \in \mathbf B(M)$, the fourientation $\fourmap{B_1}{\sigma}\cap-\fourmap{B_2}{\sigma}$ is not compatible with any $\vec C \in \SignedCircuits M$ (resp. $\fourmap{B_1}{\sigma^*}\cap-\fourmap{B_2}{\sigma^*}$ is not compatible with any $\vec C^* \in \SignedCoCircuits M$). 


A circuit-cocircuit signature $(\sigma, \sigma^*)$ is \emph{triangulating} if $\sigma$ and $\sigma^*$ are both triangulating. 
\end{defn} 

\begin{remark}
The definition for the triangulating signatures is different from the one in \cite[Definition 1.14]{Ding2}. By \cite[Theorem 1.16]{Ding2}, the two definitions are equivalent. For this paper, the current version will make our theory more concise. See~\cite[Lemma 4.16]{Ding2} for a proof that triangulating signatures exist for any regular matroid. 
\end{remark}

\begin{remark}\label{rem:tri_vs_acyc}
The paper~\cite{BBY} considers a subclass of triangulating circuit/cocircuit signatures called \emph{acyclic} circuit/ cocircuit signatures. See Section~\ref{sec:acylic} for more discussion on various conditions for signatures.

\end{remark}



\begin{remark} \label{rem:planar}
Let us mention an important example for a triangulating circuit signature, which is elaborated on in Section~\ref{sec:planar}. For a plane graph, simple cycles (as the  circuits of the matroid) oriented counterclockwise form a triangulating circuit signature (\cite{YCH}, cf. Remark \ref{rem:tri_vs_acyc}). 
Also, for a connected plane graph, minimal cuts (as the cocircuits of the matroid) ``oriented away'' from a fixed vertex $v$ form a triangulating cocircuit signature. In fact, such a cocircuit signature is the circuit signature of the dual graph as described above(with the dual face of $v$ embedded as the outermost face). 
Notice that the circuit and cocircuit signatures given in Example \ref{ex:signature} fall into the above cases. Hence, Example \ref{ex:signature} gives a triangulating circuit-cocircuit signature. 
\end{remark}


\begin{defn} Let $M$ be a regular matroid and $(\sigma,\sigma^*)$ be a circuit-cocircuit signature. An orientation $\ori O$ is \emph{$\sigma$-compatible} (resp. \emph{$\sigma^*$-compatible}) if every signed circuit (resp. cocircuit) compatible with $\ori O$ is in $\sigma$ (resp. $\sigma^*$). An orientation is \emph{$(\sigma,\sigma^*)$-compatible} if it is both $\sigma$-compatible and $\sigma^*$-compatible.  
\end{defn}

\begin{prop}\cite[Proposition~1.19(1)]{Ding2}\label{prop:uniquemin}
Fix a regular matroid $M$ and let $(\sigma,\sigma^*)$ be a triangulating circuit-cocircuit signature. Then each equivalence class in $\ccequiv(M)$ contains a unique $(\sigma,\sigma^*)$-compatible orientation. 
\end{prop}

Proposition~\ref{prop:uniquemin} implies that the following notion is well-defined. 
\begin{defn}\label{def:minimize}Given an orientation $\ori O$, let $\ori {O}^\circ$ be the $(\sigma,\sigma^*)$-compatible orientation in the same reversal class as $ \ori O$. Likewise, given $[\ori O] \in \ccequiv(M)$, let $[\ori O]^\circ = \ori O^\circ$. Furthermore, let $\ccmin(M)$ be the set of all $(\sigma,\sigma^*)$-compatible orientations.
\end{defn}


Note that $\ori {O}^\circ$ and $\ccmin(M)$ both depend on the choice of circuit-cocircuit signature. We omit a reference to this signature in the notation for readability. 
By Proposition~\ref{prop:uniquemin}, we can also define the canonical action on the set $\ccmin(M)$. 

\begin{defn}\label{def:canonical_action_for_representatives}
Given a triangulating circuit-cocircuit signature, for $s\in S(M)$ and $\ori O \in \ccmin(M)$, we define 
\[s \cdot \ori O := (s \cdot [\ori O])^\circ.\]    
\end{defn}

\begin{ex}\label{ex:canonical_action_with_representatives}
    Let us return to Example \ref{ex:canonical_action_without_representatives}. There, we noted that $[\vec {f_1}]\cdot [(-,-,+,+)]=[(+,-,+,+)]$ and $[\vec {f_3}]\cdot [(-,-,+,+)]=[(+,-,+,-)]$. Note that with respect to the circuit-cocircuit signature $(\sigma,\sigma^*)$ from Example \ref{ex:signature}, the orientations $(-,-,+,+)$ and $(+,-,+,+)$ are $(\sigma,\sigma^*)$-compatible, while $(+,-,+,-)$ is not, since $\vec {f_3}- \vec {f_4} \sim (+,-,+,-)$ but $\vec {f_3}- \vec {f_4}$ is a signed circuit that is not in $\sigma$. Nevertheless, one can show that after reversing this circuit, the orientation is $(\sigma,\sigma^*)$-compatible. In particular, we have $(+,-,+,-)^\circ = (+,-,-,+)$, and it follows that $[\vec {f_3}]\cdot (-,-,+,+) = (+,-,-,+)$.
\end{ex}

Now we are in the position to introduce the BBY bijection.  

\begin{defn}[BBY bijection]
Fix a regular matroid $M$ and a pair $(\sigma,\sigma^*)$ of triangulating signatures. The map $\BBY_{(M,\sigma,\sigma^*)}:
\mathbf{B}(M)\rightarrow\Orientations{M}$ is given by $B\mapsto \fourmap B \sigma \cap \fourmap {B}{\sigma^*}$.
\end{defn}



\begin{thm}\cite{BSY, Ding2}\label{thm:BBYbij}
Fix a regular matroid $M$ and a pair $(\sigma,\sigma^*)$ of triangulating signatures. The BBY map $\BBY_{(M,\sigma,\sigma^*)}$ is a bijection between $\mathbf{B}(M)$ and $\ccmin(M)$. In particular, this map induces a bijection between $\mathbf{B}(M)$ and $\ccequiv(M)$.
\end{thm}

\begin{defn}[BBY action]
    The pullback of the canonical action on the circuit-cocircuit minimal orientations to the bases by the bijection $\BBY_{(M,\sigma,\sigma^*)}$ is called the {\em BBY action}. To be precise, the action of $g\in S(M)$ on the basis $\BBY_{(M,\sigma,\sigma^*)}^{-1}(\ori O)\in \mathbf B(M)$ is defined as \[g \cdot \BBY_{(M,\sigma,\sigma^*)}^{-1}(\ori O) := \BBY_{(M,\sigma,\sigma^*)}^{-1}((g \cdot
[\ori O])^\circ).\]
\end{defn}

\begin{thm}\cite{BSY, Ding2}\label{thm:BBY_simply_transitive}
The BBY action is a simply transitive group action of $S(M)$ on $\mathbf B(M)$.
\end{thm}
\begin{proof}
    This follows immediately from Theorems~\ref{thm:simplytransitive} and~\ref{thm:BBYbij}.
\end{proof}

\begin{ex}\label{ex:BBY_action}
Take the graphic matroid $M$ of Figure \ref{fig:Changxin's_graph} with the circuit-cocircuit signature $(\sigma,\sigma^*)$ of Example \ref{ex:signature}. Let $T$ be the spanning tree with edge set $\{f_1,f_3\}$. Then $\BBY_{(M,\sigma,\sigma^*)}(T) = (-,-,+,+)$.

Let us compute the BBY action of $\vec {f_1}$ on $T$.
As we computed in Example \ref{ex:canonical_action_with_representatives}, $[\vec {f_1}]\cdot (-,-,+,+) = (+,-,+,+)$. Hence we need to find the preimage of $(+,-,+,+)$ at the BBY bijection. One can check that for $T'=\{f_2,f_3\}$, we have $\BBY_{M,\sigma,\sigma^*}(T')=(+,-,+,+)$. Hence $[\vec {f_1}]\cdot \BBY_{(M,\sigma,\sigma^*)}(T)=\BBY_{(M,\sigma,\sigma^*)}(T')$. 
\end{ex}

\subsection{Deletion and Contraction of Matroids and Signatures}\hspace{1pt}


\begin{defn}
An element $e\in E$ is a {\em loop} if it is contained in no basis of $M$; $e$ is a {\em coloop} if it is contained in every basis of $M$.

If $e$ is not a coloop, then the {\em deletion} of $e$ from $M$ is the matroid $$M\setminus e:=(E\setminus e,\{B\in\mathbf{B}(M): e\not\in B\}).$$ 

If $e$ is not a loop, then the {\em contraction} of $e$ from $M$ is the matroid $$M/e:=(E\setminus e,\{B\setminus e: B\in\mathbf{B}(M), e\in B\}).$$ 
\end{defn}

For the rest of this paper, we always assume $e$ is not a coloop (respectively, loop) when we are deleting (respectively, contracting) $e$ from $M$. Note that for Theorem~\ref{thm:consistency}, we only delete $e$ if it is in the complement of a basis (meaning it cannot be a coloop), and we only contract $e$ if it is in a basis (meaning it cannot be a loop). 

We can also perform deletion/contraction in a way that respects the oriented matroid structure. 



\begin{defn}\label{def:OM_minors}\cite[Proposition~3.3.1 \&~3.3.2]{Oriented}
    Let $M$ be a regular matroid with a fixed oriented matroid structure and let $e$ be an element. We have the following deletion and contraction operations at an oriented matroid level.
$$\vec{\mathbf C}(M\setminus e)=\{\vec C\setminus e: \vec C\in\vec{\mathbf C}(M), \vec C\oriat{e}=0\},$$
and $\vec{\mathbf C}(M/e)$ is the set of 1-chains with inclusionwise minimal support in
    $$\{\vec C\setminus e: \vec C\in\vec{\mathbf{C}}(M)\}.$$
Dually, 
$$\vec{\mathbf {C^*}}(M/e)=\{\vec C^*\setminus e: \vec C^*\in\vec{\mathbf{C^*}}(M), \vec C^*\oriat{e}=0\},$$
and $\vec{\mathbf{C^*}}(M\setminus e)$ is the set of 1-chains with inclusionwise minimal support in 
    $$\{\vec C^*\setminus e: \vec C^*\in\vec{\mathbf{C^*}}(M)\}.$$
\end{defn}
We will later need the following lemma which describes the relationship between signed cocircuits in $M$ and $M\setminus e$. 
\begin{lemma}\label{lem:breaking cocircuit}
Let $\vec C^*$ be a signed cocircuit of $M$ and $e$ be a non-loop edge. Then $\vec C^*\setminus e$ is either a signed cocircuit of $M\setminus e$ or a disjoint union of signed cocircuits of $M\setminus e$. 
\end{lemma}
\begin{proof} 
Since $M\setminus e$ can be realized by $A$ with the $e$-th column removed, $\vec {C^*}\setminus e$ is in the row space of $A'$. By Lemma~\ref{lem:decompose}, $\vec {C^*}\setminus e$ can be decomposed into sum of signed cocircuits $\vec{C_1^*},\ldots,\vec {C_t^*}\in\SignedCoCircuits{M\setminus e}$ of disjoint support.
\end{proof}

\begin{remark}
    There is also a dual version of Lemma~\ref{lem:breaking cocircuit} which relates signed circuits of $M$ and $M/e$. However, we only ever need the cocircuit version for our arguments. 
\end{remark}




We can also define deletion and contraction for signatures.
\begin{defn}\label{def:signature_in_a_minor} 
Let $M$ be a regular matroid and $e$ an element of $E$.
For a circuit signature $\sigma$, we define $\sigma\setminus e:=\{\vec C\setminus e: \vec C\in\sigma\}\cap\vec{\mathbf C}(M\setminus e)$ and $\sigma/e:=\{\vec C\setminus e: \vec C\in\sigma\}\cap\vec{\mathbf C}(M/e)$.
Define the deletion and contraction of a cocircuit signature similarly. 


\end{defn}

It follows from Definition~\ref{def:OM_minors} that the new collection also contains, for each circuit (respectively, cocircuit) of the deletion or contraction, exactly one signed circuit (respectively, signed cocircuit) supported on it.
In particular, the deletion or contraction of a signature is still a signature of the deletion or contraction of the matroid.



We can also define deletion and contraction for a fourientation. 

\begin{defn}\label{def:fourientation c/d}
Let $M$ be a regular matroid, $\fouri{F}$ be a fourientation of $M$, and $e$ be an element of the ground set. We denote by $\fouri{F}\setminus e$ the fourientation of $M\setminus e$ obtained by restricting $\fouri{F}$ from the ground set of $M$ to the ground set of $M\setminus e$. Define $\fouri{F}/e$ similarly.  
\end{defn}

We have the following relationships between these notions.

\begin{lemma}\label{lem:bases c/d}
Let $M$ be a regular matroid and $(\sigma, \sigma^*)$ be a circuit-cocircuit signature. Fix $B \in \mathbf B(M)$.
\begin{enumerate}
    \item For $e \in B$, we have 
    \[\fouri{F}(B\setminus e,\sigma/e)=\fouri{F}(B,\sigma)/e \hspace{.5 cm} \text{ and } \hspace{.5 cm} \fouri{F}(B\setminus e,\sigma^*/e)=\fouri{F}(B,\sigma^*)/e.\]
    \item For $e \in E \setminus B$, we have
    \[\fouri{F}(B,\sigma\setminus e)=\fouri{F}(B,\sigma)\setminus e\hspace{.5 cm} \text{ and } \hspace{.5 cm} \fouri{F}(B,\sigma^*\setminus e)=\fouri{F}(B,\sigma^*)\setminus e\]
\end{enumerate}
\end{lemma}

\begin{proof}
We will only prove for $\sigma$, as the other half follows by duality. 

When $e\in B$, it is not hard to prove that the map $\vec C\mapsto\vec C\setminus e$ is a bijection between the signed fundamental circuits of $M$ and those of $M/e$. The assertion 
\[\fouri{F}(B\setminus e,\sigma/e) =\fouri{F}(B,\sigma)/e \]
is equivalent to the assertion that for any edge $f\in E\setminus B$, 
\[\fouri{F}(B\setminus e,\sigma/e)\oriat{f} =(\fouri{F}(B,\sigma)/e)\oriat{f} \]
because the edges in $B$ are all bioriented. The latter identity means for the fundamental circuit $C_f$ of $M$, the signed circuit $(\sigma/e)(C_f\setminus e)$ of $M/e$ equals $\sigma(C_f)\setminus e$, which holds due to Definition~\ref{def:signature_in_a_minor}. 

When $e\notin B$, it is not hard to prove that the map $\vec C\mapsto\vec C\setminus e$ is a bijection between the signed fundamental circuits of the edges other than $e$ in $M$ and the signed fundamental circuits of $M\setminus e$. (In this case, $\vec C=\vec C\setminus e$.) The assertion 
\[\fouri{F}(B,\sigma\setminus e)=\fouri{F}(B,\sigma)\setminus e\]
is equivalent to that for any edge $f\in E\setminus (B\cup e)$, 
\[\fouri{F}(B,\sigma\setminus e)\oriat{f}=(\fouri{F}(B,\sigma)\setminus e)\oriat{f}. \]
The latter identity means for the fundamental circuit $C_f$ of $M$, the signed circuit $(\sigma\setminus e)(C_f\setminus e)$ of $M\setminus e$ equals $\sigma(C_f)\setminus e$, which holds due to Definition~\ref{def:signature_in_a_minor}. 

\end{proof}

Using Lemma~\ref{lem:bases c/d}, we can easily prove the following two lemmas. 


\begin{lemma}
Let $\sigma$ be a triangulating circuit signature, $\sigma^*$ be a triangulating cocircuit signature, and $e$ be an element. Then $\sigma\setminus e$ and $\sigma/e$ are triangulating circuit signatures and $\sigma^*\setminus e$ and $\sigma^*/e$ are triangulating cocircuit signatures.
\end{lemma}

\begin{proof} 
We only prove that $\sigma\setminus e$ is triangulating, as the other cases are similar.
Take any bases $B_1$ and $B_2$ of $M\setminus e$. Then, $B_1$ and $B_2$ are also bases of $M$, and they do not contain $e$.
By Definition~\ref{def:triangulating}, $\fouri{F}:=\fouri{F}(B_1,\sigma)\cap-\fouri{F}(B_2,\sigma)$ is not compatible with any signed circuits of $M$. By Lemma~\ref{lem:bases c/d}, $\fouri{F}(B_1,\sigma\setminus e)\cap-\fouri{F}(B_2,\sigma\setminus e)=\fouri{F}\setminus e$, which certainly is not compatible with any signed circuits of $M\setminus e$.
By Definition~\ref{def:triangulating}, it follows that $\sigma\setminus e$ is triangulating.
\end{proof}

The next lemma establishes the relation between the BBY bijection of a regular matroid $M$ and the BBY bijection of $M\setminus e$ or $M/e$.   
\begin{lemma} \label{lem:BBY_delcont}
Let $B$ be a basis of $M$ and $e \in E$. 
\begin{enumerate}
    \item If $e \in E \setminus B$, then for all $x \in E \setminus e$, we have \[\BBY_{(M\setminus e,\sigma\setminus e,\sigma^*\setminus e)}(B)\oriat{x}=\BBY_{(M,\sigma,\sigma^*)}(B)\oriat{x}.\]
    \item If $e \in B$, then for all $x \in E \setminus e$, we have \[\BBY_{(M/ e,\sigma/ e,\sigma^*/e)}(B\setminus e)\oriat{x}=\BBY_{(M,\sigma,\sigma^*)}(B)\oriat{x}.\]
\end{enumerate}
\end{lemma} 
\begin{proof}

By definition, we have\[\BBY_{(M\setminus e,\sigma\setminus e,\sigma^*\setminus e)}(B)=\fouri{F}(B,\sigma\setminus e)\cap\fouri{F}(B,\sigma^*\setminus e)\]and\[\BBY_{(M,\sigma,\sigma^*)}(B)=\fouri{F}(B,\sigma)\cap\fouri{F}(B,\sigma^*).\]

By Lemma~\ref{lem:bases c/d}, part (1) holds. Part (2)
can be proved similarly. 

\end{proof}

\subsection{The Main Theorem}

Let us restate our main theorem, and give an illustrating example.

\begin{customthm}{\ref{thm:consistency}}
    Let $M$ be a regular matroid and $(\sigma,\sigma^*)$ be a triangulating circuit-cocircuit signature. Suppose that $\vec f$ is an arc and $B_1,B_2 \in \mathbf B(M)$ such that
    \[[\vec f] \cdot \BBY_{(M,\sigma,\sigma^*)}(B_1)=\BBY_{(M,\sigma,\sigma^*)}(B_2).\] Then, the following 3 properties must hold. 
\begin{enumerate}
    \item For any $e \in (B_1^c \cap B_2^c)\setminus f$, we have
    \[ [\vec f] \cdot \BBY_{(M\setminus e,\sigma\setminus e,\sigma^*\setminus e)}(B_1)=\BBY_{(M\setminus e,\sigma\setminus e,\sigma^*\setminus e)}(B_2).\]
    \item For any $e \in (B_1 \cap B_2)\setminus f$, we have
    \[ [\vec f] \cdot \BBY_{(M/ e,\sigma/ e,\sigma^*/ e)}(B_1\setminus e)=\BBY_{(M/ e,\sigma/ e,\sigma^*/ e)}(B_2\setminus e).\]
    \item For any $e\in E$ that is in a different connected component of $M$ than $f$, we have  
    \[e \in B_1 \iff e \in B_2.\]
\end{enumerate}
\end{customthm}


In Section~\ref{sec:torsors}, we will discuss the motivation for Theorem~\ref{thm:consistency}. In particular, this result shows that the BBY algorithm is a \emph{consistent sandpile torsor algorithm for oriented regular matroids with triangulating circuit-cocircuit signatures}. 

\begin{ex}\label{ex:for_main_theorem}
    Take the graphic matroid $M$ from Figure \ref{fig:Changxin's_graph} with the cycle-cocycle signature $(\sigma,\sigma^*)$ from Example \ref{ex:signature}. 
    The first row of Figure \ref{fig:main_thm_ex} demonstrates Theorem~\ref{thm:consistency}(1) and the second row demonstrates Theorem~\ref{thm:consistency}(2).
    The depicted orientations are the circuit-cocircuit minimal orientations assigned to the spanning trees by the BBY bijection.

    Let us explain the first row.
    The upper left panel shows the action of $\vec {f_1}$ on the basis $\{f_1,f_3\}$: The action produces $\{f_2, f_3\}$ as explained in Example \ref{ex:BBY_action}.
    We have $\sigma\setminus f_4 = \{\vec {f_1}-\vec {f_2} +\vec {f_3}\}$, and $\sigma^*\setminus f_4 = \{-\vec {f_1} + \vec {f_3}, -\vec {f_1} -\vec {f_2}, \vec {f_2} + \vec {f_3}\}$.
    One can check that indeed $[\vec {f_1}] \cdot \BBY_{(M\setminus f_4, \sigma\setminus f_4, \sigma^*\setminus f_4)}(\{f_1,f_3\})= \BBY_{(M\setminus f_4, \sigma\setminus f_4, \sigma^*\setminus f_4)}(\{f_2,f_3\})$. 
\end{ex}

\begin{figure}
    \begin{center}
    \begin{tikzpicture}[scale=.8]
    \node[fill = none] at (-6,1) {$(1)$};
    \node[fill = none] at (-6,-2) {$(2)$};
    
    \draw (-5.2,2.2) -- (1.2,2.2) -- (1.2,-.6) -- (-5.2,-.6) -- cycle ;
    \draw (3.8,2.2) -- (10.3,2.2) -- (10.3,-.6) -- (3.8,-.6) -- cycle ;

    \draw (-5.2,-.8) -- (1.2,-.8) -- (1.2,-3.6) -- (-5.2,-3.6) -- cycle ;
    \draw (3.8,-.8) -- (10.3,-.8) -- (10.3,-3.6) -- (3.8,-3.6) -- cycle ;
     
    \tikzstyle{o}=[circle,fill,scale=.5,draw]
    \begin{scope}[shift={(-5,0)}]
	\node [o] (1) at (0,0) {};
	\node [o] (2) at (1,1.8) {};
	\node [o] (3) at (2,0) {};
	\draw [thick,<-,>=stealth'] (2) to node[fill = white,inner sep=1pt]{\footnotesize $f_1$} (1);
	\draw [thick,->,>=stealth',bend left=30] (1) to node[fill = white,inner sep=1pt]{\footnotesize $f_3$} (3);
	\draw [thick,->,>=stealth',bend right=30,dashed] (1) to node[fill = white,inner sep=1pt]{\footnotesize $f_4$} (3);
	\draw [thick,<-,>=stealth',dashed] (2) to node[fill = white,inner sep=1pt]{\footnotesize $f_2$} (3);
 \draw[{-{Stealth[length = 2mm, width = 2mm, open]}}] (2.5,.5) -- (3.5,.5);
 \node[fill = none] at (3,.95) {\small $[\vec {f_1}] \cdot$};

    \end{scope}
    
    \begin{scope}[shift={(-1,0)}]
	\node [o] (1) at (0,0) {};
	\node [o] (2) at (1,1.8) {};
	\node [o] (3) at (2,0) {};
	\draw [thick,->,>=stealth',dashed] (2) to node[fill = white,inner sep=1pt]{\footnotesize $f_1$} (1);
	\draw [thick,->,>=stealth',bend left = 30] (1) to node[fill = white,inner sep=1pt]{\footnotesize $f_3$} (3);
	\draw [thick,->,>=stealth',bend right = 30, dashed] (1) to node[fill = white,inner sep=1pt]{\footnotesize $f_4$} (3);
	\draw [thick,<-,>=stealth'] (2) to node[fill = white,inner sep=1pt]{\footnotesize $f_2$} (3);

  \draw[{-{Stealth[length = 3mm, width = 2.5mm, open,color = blue]},color = blue}] (2.5,.5) -- (4.5,.5);
 \node[fill = none] at (3.5,.95) {\color{blue}{\small{Delete $f_4$}}};
    \end{scope}
    
    \begin{scope}[shift={(4,0)}]
	\node [o] (1) at (0,0) {};
	\node [o] (2) at (1,1.8) {};
	\node [o] (3) at (2,0) {};
	\draw [thick,<-,>=stealth'] (2) to node[fill = white,inner sep=1pt]{\footnotesize $f_1$} (1);
	\draw [thick,->,>=stealth',bend left=10] (1) to node[fill = white,inner sep=1pt]{\footnotesize $f_3$} (3);
	\draw [thick,<-,>=stealth',dashed] (2) to node[fill = white,inner sep=1pt]{\footnotesize $f_2$} (3);
 \draw[{-{Stealth[length = 2mm, width = 2mm, open]}}] (2.5,.5) -- (3.5,.5);
 \node[fill = none] at (3,.95) {\small $[\vec {f_1}] \cdot$};
    \end{scope}
    
    \begin{scope}[shift={(8,0)}]
	\node [o] (1) at (0,0) {};
	\node [o] (2) at (1,1.8) {};
	\node [o] (3) at (2,0) {};
	\draw [thick,->,>=stealth',dashed] (2) to node[fill = white,inner sep=1pt]{\footnotesize $f_1$} (1);
	\draw [thick,->,>=stealth',bend left = 10] (1) to node[fill = white,inner sep=1pt]{\footnotesize $f_3$} (3);
	\draw [thick,<-,>=stealth'] (2) to node[fill = white,inner sep=1pt]{\footnotesize $f_2$} (3);
    \end{scope}

    \begin{scope}[shift={(-5,-3)}]
	\node [o] (1) at (0,0) {};
	\node [o] (2) at (1,1.8) {};
	\node [o] (3) at (2,0) {};
    \draw [thick,<-,>=stealth'] (2) to node[fill = white,inner sep=1pt]{\footnotesize $f_1$} (1);
	\draw [thick,<-,>=stealth',dashed,bend left = 30] (1) to node[fill = white,inner sep=1pt]{\footnotesize $f_3$} (3);
	\draw [thick,->,>=stealth',bend right = 30] (1) to node[fill = white,inner sep=1pt]{\footnotesize $f_4$} (3);
	\draw [thick,<-,>=stealth',dashed] (2) to node[fill = white,inner sep=1pt]{\footnotesize $f_2$} (3);
 \draw[{-{Stealth[length = 2mm, width = 2mm, open]}}] (2.5,.5) -- (3.5,.5);
 \node[fill = none] at (3,.95) {\small $[\vec {f_1}] \cdot$};
    \end{scope}

    \begin{scope}[shift={(-1,-3)}]
	\node [o] (1) at (0,0) {};
	\node [o] (2) at (1,1.8) {};
	\node [o] (3) at (2,0) {};
 
	\draw [thick,->,>=stealth',dashed] (2) to node[fill = white,inner sep=1pt]{\footnotesize $f_1$} (1);
	\draw [thick,<-,>=stealth',dashed,bend left = 30] (1) to node[fill = white,inner sep=1pt]{\footnotesize $f_3$} (3);
	\draw [thick,->,>=stealth',bend right = 30] (1) to node[fill = white,inner sep=1pt]{\footnotesize $f_4$} (3);
	\draw [thick,<-,>=stealth'] (2) to node[fill = white,inner sep=1pt]{\footnotesize $f_2$} (3);

 \draw[{-{Stealth[length = 3mm, width = 2.5mm, open,color = red]},color = red}] (2.5,.5) -- (4.5,.5);
 \node[fill = none] at (3.5,.95) {\color{red}{\small{Contract $f_4$}}};
    \end{scope}

    \begin{scope}[shift={(5.1,-3)}]
	\node [o] (1) at (0,0) {};
	\node [o] (2) at (0,1.8) {};
	\draw [thick,<-,>=stealth',bend right = 90,looseness=1.5] (2) to node[fill = white,inner sep=1pt]{\footnotesize $f_1$} (1);
	\draw [thick,->,>=stealth',out=60,in=120,looseness=30,dashed] (1) to node[fill = white,inner sep=1pt]{\footnotesize $f_3$} (1);
	\draw [thick,<-,>=stealth',bend left=90,dashed,looseness=1.5] (2) to node[fill = white,inner sep=1pt]{\footnotesize $f_2$} (1);
 \draw[{-{Stealth[length = 2mm, width = 2mm, open]}}] (1.4,.5) -- (2.4,.5);
 \node[fill = none] at (1.9,.95) {\small $[\vec {f_1}] \cdot$};
    \end{scope}

    \begin{scope}[shift={(9.1,-3)}]
	\node [o] (1) at (0,0) {};
	\node [o] (2) at (0,1.8) {};
	\draw [thick,->,>=stealth',bend right = 90,dashed,looseness=1.5] (2) to node[fill = white,inner sep=1pt]{\footnotesize $f_1$} (1);
	\draw [thick,->,>=stealth',out=60,in=120,looseness=30,dashed] (1) to node[fill = white,inner sep=1pt]{\footnotesize $f_3$} (1);
	\draw [thick,<-,>=stealth',bend left=90,looseness=1.5] (2) to node[fill = white,inner sep=1pt]{\footnotesize $f_2$} (1);
    \end{scope}
\end{tikzpicture}
\end{center}
\caption{Above are illustrations for the first two parts of Theorem \ref{thm:consistency}. See Example \ref{ex:for_main_theorem} for details.}
\label{fig:main_thm_ex}
\end{figure}
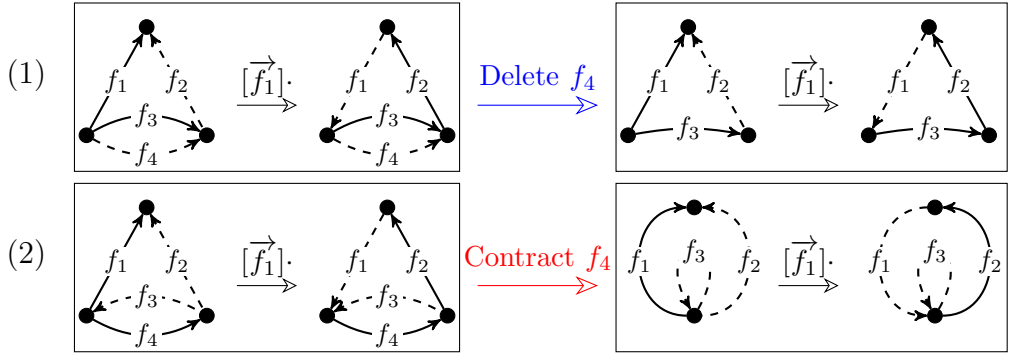

\section{Sandpile Torsor Algorithms, Consistency, and Special Cases} \label{sec:torsors}

In this section, we elaborate on the motivation for Theorem~\ref{thm:consistency}, which is the main result of the paper. Then we will see some applications of Theorem~\ref{thm:consistency} to the special cases of acyclic signatures and planar graphs. However, none of the material in this section will be used in our proof of Theorem~\ref{thm:consistency}, so readers primarily interested in the proof can safely skip to Section~\ref{sec:action}.

\subsection{Sandpile Torsors}\label{sec:torsordefs}

The term \emph{sandpile torsor} was coined by Ellenberg in a MathOverflow post~\cite{JSE}. Later, Ganguly and the second author defined \emph{sandpile torsor actions} and \emph{sandpile torsor algorithms} on plane graphs. Here, we will generalize their definition to apply to a larger class of objects. 

In particular, we will work with a \emph{minor-closed class of regular matroids with auxiliary data}. More precisely, let $\mathbf M$ be a set of regular oriented matroids where each $M \in \mathbf M$ is associated with a class of auxiliary data, which we denote $\mathbf A(M)$. We need the following properties to be satisfied:
\begin{itemize}
    \item For any $M \in \mathbf M$ and $e$ in the ground set of $M$, we must have $M\setminus e \in \mathbf M$ and $M /e \in \mathbf M$. 
    \item For any $M \in \mathbf M$, $\mathcal A \in \mathbf A(M)$, and $e$ in the ground set of $M$, there must be a natural way to define $\mathcal A\setminus e \in \mathbf A(M\setminus e)$ and $\mathcal A / e \in \mathbf A(M/e)$. 
\end{itemize}
Below are three examples of minor-closed classes of regular matroids with auxiliary data that can be used as the setting for sandpile torsors. 

\begin{enumerate}
    \item Let $\mathbf M$ be the set of all regular oriented matroids. For $M \in \mathbf M$, let $\mathbf A(M)$ be the set of triangulating circuit-cocircuit signatures of $M$. This has been the primary setting for this paper. 
    \item Let $\mathbf M$ be the set of all regular oriented matroids. For $M \in \mathbf M$, let $\mathbf A(M)$ be the set of \emph{acyclic} circuit-cocircuit signatures of $M$ (defined below). This is the setting used in~\cite{BBY} and we will discuss it further in Section~\ref{sec:acylic}.  
    \item Let $\mathbf M$ be the set of oriented graphic matroids for \emph{planar graphs}. For $M \in \mathbf M$, let $\mathbf A(M)$ be the set of possible planar embeddings of the graph associated with $M$. This is the setting used in~\cite{GM} and we will discuss it further in Section~\ref{sec:planar}.
\end{enumerate}

We remark that further examples can be constructed by restricting (1) or (2) to any minor-closed family of regular oriented matroids. For example, the case of graphic matroids were implicitly studied in \cite{YCH}. (3) can also be further restricted to a minor-closed family of planar graphs by mirroring the minor operation to planar embeddings. 

Recall that for any regular matroid $M$, the size of the set of bases $\mathbf B(M)$ is always the same as the size of the sandpile group $S(M)$. 

\begin{defn}\label{def:STaction} 
    Let $M$ be a particular oriented regular matroid along with some auxiliary structure. A \emph{sandpile torsor action} is a choice of simply transitive action of the sandpile group $S(M)$ on the bases $\mathbf B(M)$ given any $\mathcal A \in \mathbf A(M)$. 
\end{defn}

We remark that it is hopeless to give canonical simply transitive actions of $S(M)$ on $\mathbf B(M)$ without some kind of auxiliary data, see~\cite[Remark 2.11]{GM}.


\begin{defn}\label{def:STA}A \emph{sandpile torsor algorithm} for a minor-closed class of regular matroids with auxiliary data is a choice of sandpile torsor action for each $M\in \mathbf M$ and $\mathcal{A}\in\mathbf{A}(M)$. 
\end{defn}

In other words, a \emph{sandpile torsor action} is a simply transitive action of the sandpile group on the bases for a \emph{particular} regular matroid (with auxiliary data), while a \emph{sandpile torsor algorithm} is a general construction for every possible $M \in \mathbf M$ and $\mathcal A \in \mathbf A(M)$. 

Given a sandpile torsor algorithm $\Gamma$, we will write $\Gamma_{(M,\mathcal A)}$ for the sandpile torsor action on a specific matroid $M$ with auxiliary data $\mathcal A$. In particular, for any $s \in S(M)$ and $B, B' \in \mathbf B(M)$, the equality 
\[\Gamma_{(M,\mathcal A)}(s,B) = B'\]
says that $s$ maps $B$ to $B'$ with respect to the sandpile torsor action $\Gamma_{(M,\mathcal A)}$.

For the auxiliary structure of circuit cocircuit signatures, we investigate the following sandpile torsor algorithm in this paper.

\begin{defn}\label{def:BBYSTA}
For a regular matroid $M$, a triangulating circuit-cocircuit signature $(\sigma,\sigma^*)$, $s \in S(M)$, and $B\in \mathbf B(M)$, define
\[ \Gamma^\beta_{(M,\sigma,\sigma^*)}(s,B) := \beta_{(M,\sigma,\sigma^*)}^{-1}(s \cdot \beta_{(M,\sigma,\sigma^*)}(B)).\]
We call $\Gamma^\beta$ the \emph{BBY sandpile torsor algorithm (with the data of triangulating circuit-cocircuit signatures)}.
\end{defn} 


It is not hard to see that the construction from Definition~\ref{def:BBYSTA} can be generalized to other families of bijections between $\mathbf B(M)$ and $\ccmin(M)$. In particular, suppose that for each regular matroid $M$ and triangulating circuit cocircuit signature $(\sigma,\sigma^*)$, we have a bijection $\alpha_{(M,\sigma,\sigma^*)}: \mathbf B(M) \to \ccmin(M)$. Then, we obtain a sandpile torsor algorithm $\Gamma^\alpha$ by setting 

\begin{equation}\label{eq:Bij_to_torsor}\Gamma^\alpha_{(M,\sigma,\sigma^*)}(s,B) := \alpha_{(M,\sigma,\sigma^*)}^{-1}(s \cdot \alpha_{(M,\sigma,\sigma^*)}(B)).\end{equation}

A natural question is whether every sandpile torsor algorithm for regular matroids (with the data of triangulating circuit-cocircuit signatures) can be constructed from a family of bijections using~\eqref{eq:Bij_to_torsor}. In the following theorem, we show that this is indeed the case (despite the fact that Definition~\ref{def:STaction} makes no mention of the canonical action). 

\begin{thm}\label{thm:Torsor_to_bij}
Let $\Gamma$ be a sandpile torsor algorithm for regular matroids (with the data of triangulating circuit-cocircuit signatures). For any regular matroid $M$ and triangulating circuit-cocircuit signature $(\sigma,\sigma^*)$, there exists a bijection $\alpha_{(M,\sigma,\sigma^*)}: \mathbf B(M) \to \ccmin(M)$ such that $\Gamma_{(M,\sigma,\sigma^*)}(s,B) = \alpha_{(M,\sigma,\sigma^*)}^{-1}(s \cdot \alpha_{(M,\sigma,\sigma^*)}(B))$. More precisely, there are exactly $|\mathbf B(M)|$ such bijections. 
\end{thm}
\begin{proof}
Fix any pair $(\ori O ,B) \in \ccmin(M)\times \mathbf B(M)$, and then set $\alpha_{(M,\sigma,\sigma^*)}(B) = \ori O$. Next, consider an element $B' \in \mathbf B(M)$. Since $\Gamma_{(M,\sigma,\sigma^*)}$ is simply transitive, there must be some $s \in S(M)$ such that $\Gamma_{(M,\sigma,\sigma^*)}(s,B) = B'$. Set $\alpha_{(M,\sigma,\sigma^*)}(B') = s\cdot \ori O$. This construction can be used to define the map $\alpha_{(M,\sigma,\sigma^*)}$ for all $B' \in \mathbf B(M)$. It follows from the fact that $\Gamma_{(M,\sigma,\sigma^*)}$ and the canonical action are simply transitive that $\alpha$ must be a bijection. Note that the bijection $\alpha_{(M,\sigma,\sigma^*)}$ satisfies $\Gamma_{(M,\sigma,\sigma^*)}(s,B) = \alpha_{(M,\sigma,\sigma^*)}^{-1}(s \cdot \alpha_{(M,\sigma,\sigma^*)}(B))$, and any desired bijection can be constructed in this way.

To count the number of bijections, notice that there are $|\mathbf B(M)|^2$ initial pairs we could have chosen, and each bijection will be produced from $|\mathbf B(M)|$ pairs. Thus, there are $|\mathbf B(M)|$ possible bijections.
\end{proof}

\subsection{A General Definition of Consistency}
The choices of sandpile torsor actions in Definition~\ref{def:STA} can be arbitrary across different matroids, hence having very little structure. In order to impose meaningful relations among the choices, Ganguly and the second author introduced the concept of \emph{consistency} of sandpile torsor algorithms. However, they primarily worked in the context of plane graphs~\cite{GM}. In this section, we give a generalization of consistency to any minor-closed class of regular matroids with auxiliary data. 

The following definition rephrases the properties of Theorem~\ref{thm:consistency} in the language of sandpile torsor algorithms. 

\begin{defn}\label{def:general_consistency}
We say that the sandpile torsor algorithm $\Gamma$ is \emph{consistent} if for any $M \in \mathbf M$, any $\mathcal A \in \mathbf A(M)$, any $B \in \mathbf B(M)$, and any arc $\vec f$, we have:
\begin{enumerate}
    \item For any $e \in (B^c \cap \Gamma_{(M,\mathcal{A})}([\vec f],B)^c)\setminus f$, we have
    \[\Gamma_{(M\setminus e,\mathcal{A}\setminus e)}([\vec f], B) = \Gamma_{(M,\mathcal{A})}([\vec f],B).\]
    \item For any $e \in (B \cap \Gamma_{(M,\mathcal{A})}([\vec f],B))\setminus f$, we have
    \[ \Gamma_{(M/ e,\mathcal{A}/e)}([\vec f], B) = \Gamma_{(M,\mathcal{A})}([\vec f],B) \setminus e.\]
    \item If $e$ and $f$ are in different connected components of $M$, then 
    \[e \in B \iff e \in \Gamma_{(M,\mathcal{A})}([\vec f],B)\]
\end{enumerate}
\end{defn}

The following result is immediate from Theorem~\ref{thm:consistency}.

\begin{cor}\label{cor:BBYconsistent} 
    The BBY sandpile torsor algorithm (with the auxiliary data of triangulating circuit-cocircuit signatures) is consistent.  
\end{cor}

In the next subsection, we will show that this result also holds when \emph{acyclic signatures} are used in place of \emph{triangulating signatures}.

\subsection{Acyclic signatures}\label{sec:acylic}

The notion of acyclic signatures was first introduced in \cite{YCH} by the fourth author for circuit signatures before being extended to circuit-cocircuit signatures in \cite{BBY}. \
The somewhat technical definition arrives naturally in the context of polyhedral geometry, which was the main theme of these two papers. 

\begin{defn}
A circuit signature $\sigma$ of a regular matroid $M$ is {\em acyclic} if the equation $\sum_{\vec C\in\sigma} \lambda_C\vec C={\bf 0}$ has no nonzero, non-negative solution $\lambda_C$'s. Define an acyclic cocircuit signature similarly.
\end{defn}

\begin{lemma}\cite{Ding2}\label{lem:acyctriang} 
    Acyclic signatures are triangulating.  
\end{lemma}

See \cite[Proposition~3.14]{Ding2} for an example of triangulating signature of a regular matroid (indeed, the graphical matroid of a planar graph) that is not acyclic.

\begin{lemma}\label{lem:acycminorclosed}
Let $\sigma$ be an acyclic circuit signature, and let $\sigma^*$ be an acyclic cocircuit signature, moreover, let $e$ be an arbitrary element of the ground set. Then $\sigma\setminus e$ and $\sigma/e$ are acyclic circuit signatures, and $\sigma^*\setminus e$ and $\sigma^*/e$ are acyclic cocircuit signatures.
\end{lemma}
\begin{proof}
First, we show that $\sigma\setminus e$ is an acyclic signature. For this, notice that any signed circuit of $M\setminus e$ is a signed circuit in $M$.
Take an arbitrary nonnegative linear combination of signed circuits $\sum_{\vec C\in\sigma\setminus e}\lambda_C \vec C$ summing to zero in $M\setminus e$. Then this is a nonnegative linear combination of signed circuits in $M$ giving zero, hence we conclude that $\lambda_C$'s are all 0. This means that $\sigma\setminus e$ is acyclic.

Next, we show that $\sigma/e$ is an acyclic signature. 
Take an arbitrary nonnegative linear combination of signed circuits $\sum_{\vec C\in\sigma/e}\lambda_C \vec C$ summing to zero in $M/e$. Recall that for each $\vec C\in\SignedCircuits{M/e}$, there is a unique signed circuit $\vec {C'}\in\SignedCircuits{M}$ such that $\vec{C'}\setminus e = \vec C$.
Now consider $\sum_{\vec C\in\sigma/e}\lambda_C \vec {C'}$, which must be everywhere zero except possibly over $e$.
But this sum is in $\ker(A)$, and by our standing assumption $e$ is not a loop, the sum must be zero, which in turn implies $\lambda_C$'s are zeros.




   The parallel statements for $\sigma^*$ can be proven by duality.
\end{proof}

Lemma~\ref{lem:acycminorclosed} implies that when the BBY sandpile torsor algorithm is restricted from triangulating signatures to acyclic signatures, it is still a sandpile torsor algorithm. 

\begin{cor}\cite[Conjecture~6.11]{GM}\label{cor:BBYacyclic} 
    The BBY sandpile torsor algorithm (with the auxiliary data of acyclic circuit-cocircuit signatures) is consistent.
\end{cor}
\begin{proof}
    It follows from Lemma~\ref{lem:acyctriang} that this result is an immediate corollary of Theorem~\ref{thm:consistency}. 
\end{proof}


\subsection{The Planar Case} \label{sec:planar}

In this subsection, we will discuss how the BBY algorithm specializes to the context of \emph{plane graphs}. In particular, we show that Theorem~\ref{thm:consistency} gives an alternate proof of~\cite[Theorem~4.6]{GM} which says that the \emph{rotor-routing} sandpile torsor algorithm on plane graphs is consistent.

Recall from Section~\ref{sec:torsordefs} that sandpile torsor algorithms are defined with respect to a minor-closed collection of regular matroids $\mathbf M$, where each $M \in \mathbf M$ is equipped with auxiliary data $\mathbf A(M)$. In the previous two subsections, we let $\mathbf M$ be the set of all regular matroids. In this subsection, $\mathbf M$ is the set of \emph{graphic matroids} of \emph{planar graphs}. For each $M\in \mathbf M$, the auxiliary data $\mathbf A(M)$ is the set of realizations of $\mathbf M$ as explicit graphs embedded in the plane. Another way to say this is that each pair $(M,\mathcal A)$ corresponds to a \emph{plane graph}, which is another name for an embedding of a graph into the plane. 

As discussed in Section~\ref{sec:torsordefs}, we can define sandpile torsor algorithms and consistency in the setting of plane graphs.  In fact, Definition~\ref{def:general_consistency} specializes to~\cite[Definition 4.3]{GM} in the plane graph context. 

\begin{remark}
    In~\cite{GM}, the sandpile group is defined as a quotient group of $\Z^V$ instead of $\Z^E$ (where $V$ is the set of vertices of the graph). When working with regular matroids, there is no well-defined analogue to vertices, so it is necessary to look at edges instead. By~\cite[Proposition 3.2.11]{McDThesis}, the two definitions give canonically isomorphic groups as long as the graph is connected.  
\end{remark}

In \cite[Example~1.1.3]{BBY}, the authors explain how circuit and cocircuit signatures can be extracted from a plane graph. In particular, given a plane graph $G$, we first choose a planar embedding of the dual graph $G^*$. Orienting all of the cycles of $G$ counterclockwise (or all clockwise) gives an acyclic (and therefore triangulating) circuit signature for $G$. We can obtain a triangulating circuit signature for $G^*$ analogously. From here, we get a triangulating cocircuit signature for $G$ from the natural bijection between oriented cycles of $G$ and oriented cuts of $G^*$. 

The previous paragraph shows how any plane graph can be reinterpreted as a regular matroid with a triangulating circuit-cocircuit signature. In particular, the BBY sandpile torsor algorithm induces a sandpile torsor algorithm on plane graphs. It follows from~\cite[Proposition 18]{YCH} that the induced sandpile torsor algorithm is equivalent to the \emph{Bernardi sandpile torsor algorithm}, and does not depend on the embedding of $G^*$. By \cite[Theorem~7.1]{BW_Bernardi}, the \emph{Bernardi sandpile torsor algorithm} is equivalent to the \emph{rotor-routing sandpile torsor algorithm} (when working on plane graphs). Thus, the rotor-routing algorithm is a special case of the BBY algorithm. 

As plane graphs are also minor-closed classes, we conclude that \cite[Theorem 4.6]{GM} is an immediate corollary of Theorem~\ref{thm:consistency}.

\section{Characterizing the Canonical Action of an Arc on $\ccmin$} \label{sec:action}

Throughout this section, we will fix a regular matroid $M$ realized by a totally unimodular matrix $A$, as well as a triangulating circuit-cocircuit signature $(\sigma,\sigma^*)$ of $M$. Our goal of this section is describe the canonical action of an arc acting on a $(\sigma,\sigma^*)$ compatible orientation of $M$ (see Definition~\ref{def:canonical_action_for_representatives}). The main result of this section is Theorem~\ref{thm:generalDescription}, which gives a detailed description of how orientations change when acted on by an arc. For example, this theorem shows that it suffices to flip at most one circuit and at most one cocircuit, as well as the arc itself. Theorem~\ref{thm:generalDescription} will be crucial for proving Theorem~\ref{thm:consistency} in Section~\ref{sec:main}, and may be of independent interest. 


\subsection{
Lemmas on Orientations and Fourientations}


We give a few useful lemmas on how to obtain $\ori O^\circ$ (recall Definition~\ref{def:minimize}) from an orientation $\ori O$. 

We use Lemma~\ref{lem:decompose} to show that one can reach a $(\sigma,\sigma^*)$-compatible orientation in a class by reversing disjoint circuits and cocircuits.
This property has been shown in multiple references implicitly such as the proof of \cite[Theorem~3.3]{GY} or \cite[Section~2.1]{Ding1}, but we give a standalone proof as it will be very important for us. 

\begin{lemma}\label{lem:disjoint}
Let $\ori O$ be an orientation of $M$. Then, there exist $\vec{C_1},\ldots,\vec{C_s}\in\SignedCircuits{M}$ and $\vec{C_1^*},\ldots,\vec{C_t^*}\in\SignedCoCircuits{M}$ (with $s,t$ possibly zero) such that:
\begin{enumerate}
\item\label{it:disj} These signed circuits and cocircuits all have mutually disjoint supports.
\item\label{it:compat} For all $i \in [t]$ and $j \in [s]$, we have $\vec{C_i}\compat \ori O$, $\vec{C^*_j} \compat \ori O$, $-\vec{C_i}\compat \ori O^\circ$, and $-\vec{C_j^*}\compat \ori O^\circ$.
\item We have $\reverse{P}\ori O = \ori O^\circ$, where $P:=(\bigcup_{i=1}^s C_i)\cup(\bigcup_{j=1}^t C^*_j)$.
\item\label{it:not_in_sigma} For all $i \in [t]$ and $j \in [s]$, we have $\vec{C_i}\not\in\sigma$ and $\vec{C_j^*}\not\in\sigma^*$.


\end{enumerate}
\end{lemma}

\begin{proof}
By Lemma~\ref{lem:circorcocirc}, $\ori O$ can be partitioned into a totally cyclic part where only circuit reversals can happen, and an acyclic part where only cocircuit reversal can happen, while neither reversals alter the partition.
In particular, to prove the lemma, it suffices to prove that any sequence of circuit reversals is equivalent to a sequence of disjoint circuit reversals (and take dual for the cocircuit counterpart).
We say 
that a circuit reversal involves a $\vec C\in\SignedCircuits{M}$ if $C$ is being reversed and $\vec C$ is compatible with the original orientation. 

By comparing $\ori O$ and $\ori O^\circ$, the sum $\vec P$ of all signed circuits involved in the original sequence must be a simple 1-chain (namely, the sum of arcs in $\ori O$ that are flipped in the end), moreover, $\vec P$ is in $\ker(A)$ as every signed circuit is.
By Lemma~\ref{lem:decompose}, $\vec P$ can be decomposed into a sum of signed circuits $\vec{C_1},\ldots,\vec{C_s}$ of disjoint support, and from the description of $\vec P$, each $\vec{C_i}$ is compatible with $\ori O$.
\end{proof}




We list a few observations that are straightforward from definition, but also frequently useful. 

\begin{lemma}\label{lem:Fouriprops} 
    Let $\fouri {F}_1, \fouri {F}_2 \in \Fourientations{M}$ and $\vec P\in \Z^E$. Then, the following properties hold. 
    \begin{enumerate}
    \item $\vec P \compat \fouri{F}_1$ if and only if $-\vec P \compat -\fouri{F}_1$.
    \item $\vec P \compat (\fouri{F}_1 \cap \fouri{F}_2)$ if and only if both $\vec P \compat \fouri {F}_1$ and $\vec P \compat \fouri {F}_2$.
    \item $\reverse{P}{(\fouri F_1 \cup \fouri F_2)} = \reverse{P}{\fouri F_1} \cup \reverse{P}{\fouri F_2}$.
    \item $\reverse{P}{(\fouri F_1 \cap \fouri F_2)} = \reverse{P}{\fouri F_1} \cap \reverse{P}{\fouri F_2}$.
    \end{enumerate}
\end{lemma}


Next, we give a more technical lemma involving the intersection of a circuit and cocircuit.

\begin{lemma}\label{lem:ori_intersection}
Suppose that $\vec e$ and $\vec f$ are distinct arcs, $\vec C$ is a signed circuit, $\vec {C^*}$ is a signed cocircuit, and $\ori O$ is an orientation. Further suppose that
\[ e\in C\cap C^*, \hspace{ .5 cm}(\vec C\setminus f) \compat \ori O, \hspace{.5 cm}\text{ and }\hspace{.5 cm}(\vec {C^*} \setminus f) \compat \ori O.\]
Then, we have $C \cap C^* = \{e,f\}$. 
\end{lemma} 

\begin{proof}

Since $\vec C\setminus f \compat \ori O$ and $\vec {C^*} \setminus f \compat \ori O$, we must have $\vec C\oriat e = \vec {C^*} \oriat e$. By Lemma~\ref{lem:ccint}, this means that there must be some $x \in C \cap C^*$ such that $\vec C\oriat x \neq \vec {C^*} \oriat x$. However, since $\vec C\setminus f \compat \ori O$ and $\vec {C^*} \setminus f \compat \ori O$, the only possibility is $x=f$. This implies that $f \in C \cap C^*$ and we can apply Lemma~\ref{lem:ccint} again to conclude that $C \cap C^* = \{e,f\}$. 
\end{proof}


\begin{lemma}\label{lem:intersection}
Suppose that $\vec e$ and $\vec f$ are distinct arcs, $\vec C$ is a signed circuit, $\vec {C^*}$ is a signed cocircuit, and $\fouri F$ is a fourientation. Further suppose that
\[ e \in C \cap {C^*}, \hspace{ .5 cm}(\vec C\setminus f) \compat \fouri F, \hspace{.5 cm}\text{ and }\hspace{.5 cm}(\vec {C^*} \setminus f) \compat (-\fouri F)^c.\]
Then, it follows that $C \cap C^* = \{e,f\}$. 
\end{lemma} 

\begin{proof}
By the compatibility conditions, we must have $\fouri{F}\oriat{x}\in\{-,+\}$ for every $x\in (C\cap C^*)\setminus f$ or else either $\fouri{F}\oriat{x}$ or $(-\fouri{F})^c\oriat{x}$ is $\emptyset$. From here, the proof is analogous to the proof of Lemma~\ref{lem:ori_intersection}.  
\end{proof}

\subsection{
Characterization of the Canonical Action}


The following straightforward lemma elaborates some details of the canonical action of $S(M)$ on $\ccmin(M)$.

\begin{lemma}\label{lem:swapf} Let $\ori O\in\ccmin(M)$ and $\vec f$ be an arc. Exactly one of the following must hold. 

\begin{enumerate}
    \item We have $\vec f \not \compat \ori O$ and \[[\vec f] \cdot \ori O = (\reverse{f}{\ori O})^\circ.\] 
    \item We have $\vec f \sim \ori O$ and there is a signed circuit $\vec C$ such that $\vec f \in \vec C$ and $\vec C \compat \ori O$. For any such $\vec C$, we have
    \[[\vec f] \cdot \ori O = (\reverse{C\setminus f}{\ori O})^\circ.\] 
    \item We have $\vec f \sim \ori O$ and there is a signed cocircuit $\vec {C^*}$ such that $\vec f \in \vec {C^*}$ and $\vec {C^*} \compat \ori O$. For any such $\vec {C^*}$, we have
    \[[\vec f] \cdot \ori O = (\reverse{C^*\setminus f}{\ori O})^\circ.\] 
\end{enumerate}

\end{lemma}


Note that in (2) and (3) the choice of $C$ or $C^*$ is not necessarily unique.
\begin{proof}
If $\vec f \not\compat \ori O$, then we can immediately apply Definition~\ref{def:groupaction} to obtain the desired conclusion. Otherwise, it follows directly by Lemma~\ref{lem:circorcocirc} that there exists a signed circuit $\vec C$ or a signed cocircuit $\vec C^*$ which satisfies the conditions of the lemma. Reversing this circuit or cocircuit gives an orientation in $[\ori O]$ that is not compatible with $\vec f$. Thus, we can again apply Definition~\ref{def:groupaction} to complete the proof. 
\end{proof}

Building off Lemma~\ref{lem:swapf}, we now give a general description of the canonical action of an arc on the set $\ccmin(M)$. 

\begin{thm} \label{thm:generalDescription} 
Let $\vec f$ be an arc and $\ori O_1, \ori O_2 \in \ccmin(M)$ such that $[\vec f] \cdot \ori O_1 = \ori O_2$. Then $\ori O_1$ can be transformed to $\ori O_2$ by the following three step process.
\begin{enumerate}
    \item Reverse at most one signed circuit or cocircuit containing $f$ that is compatible with $\ori O_1$. 
    \item Reverse $f$. 
    \item Reverse at most one signed circuit or cocircuit containing $f$ that is compatible with the new orientation.  
\end{enumerate}
Furthermore, the following conditions hold. 
\renewcommand{\theenumi}{\alph{enumi}}
\begin{enumerate}
    \item A reversal occurs during step (1) if and only if $\vec f \compat \ori O_1$. 
    \item A reversal occurs during step (3) if and only if $-\vec f \compat \ori O_2$. 
    \item If reversals occur during both step (1) and step (3), one of these is a circuit reversal while the other is a cocircuit reversal.  
    \item If reversals occur both in step (1) and step (3), the reversed circuit is $\vec C$, and the reversed cocircuit is $\vec C^*$, then we have $C\cap C^* = \{f,g\}$ for some $g \in E$. Moreover, $\ori O_2 = \reverse{(C\cup C^*)\setminus g}\ori{O}_1$. 
\end{enumerate}
\end{thm} 
\begin{proof}\mbox

\begin{enumerate}[label = \textbf{Case \arabic*},series = cases,leftmargin = *]

\item $\vec f \not\compat \ori O_1$:\\
\end{enumerate}

    It follows from Lemma~\ref{lem:swapf}(1) that $\ori O_2 = (\reverse{f}{\ori{O}_1})^\circ$. From here, it follows from Lemma~\ref{lem:disjoint} that $\ori O_2$ is obtained from $\reverse{f}{\ori{O}_1}$ by reversing a (possibly empty) collection of signed circuits and cocircuits with mutually disjoint support. Furthermore, by Lemma~\ref{lem:disjoint}(\ref{it:compat},\ref{it:not_in_sigma}), each of these signed circuits and cocircuits must be compatible with $\reverse{f}{\ori{O}_1}$, but are not in $\sigma$ or $\sigma^*$. All of these signed circuits and cocircuits must contain $f$ or else they would be compatible with $\ori{O}_1$, contradicting the fact that $\ori{O}_1$ is $(\sigma,\sigma^*)$-compatible. Furthermore, since the supports are mutually disjoint, the total number of signed circuits and cocircuits is at most 1. This proves the first half of the theorem for this case (the part before the four additional conditions). 

With respect to the theorem, we have shown that there is never a reversal in step (1) of the process, which proves condition (a) for this case. It also vacuously proves conditions (c) and (d). Finally, condition (b) holds because $f$ is reversed twice is there is a reversal at step (3), but only once if there is not. The result then follows from the fact that the final orientation of $f$ is $\ori O_2\oriat{f}$ by definition.\\ 


\begin{enumerate}[resume* = cases]
\item  $\vec f \compat \ori O_1$:\\ 
\end{enumerate}

There are two possibilities for $\ori O_2$ corresponding to the second and third case of Lemma \ref{lem:swapf}.\\

\begin{enumerate}[label = \textbf{Case 2.\arabic*},series = subcases,leftmargin = *, labelindent = \parindent]    
\item Lemma \ref{lem:swapf}(2) applies:\\
\end{enumerate} 

Lemma \ref{lem:swapf}(2) says that there exists some signed circuit $\vec C$ such that $\vec f \in  \vec C$, $\vec C \compat \ori O_1$ and $\ori O_2 = (\reverse{C\setminus f}{\ori O_1})^\circ$. Given a circuit $\vec C$ which satisfies these properties, consider the set
\[\flip(C):=\{e\in C \mid \reverse{C\setminus f}{\ori O_1}\oriat{e} \neq \ori{O}_2\oriat{e} \}.\]
In other words, $\flip(C)$ consists of the $e \in C$ whose direction has to be reversed in $\reverse{C\setminus f}{\ori O_1}$ to get $\ori O_2$. Choose a $C$ such that $\flip(C)$ is minimal among the circuits containing $f$ that are compatible with $\vec C$. In particular, we choose $\vec C$ such that there is no signed circuit $\vec C'$ satisfying $f \in C'$, $\vec {C'} \compat \ori O_1$, and $\flip(C') \subseteq \flip (C)$. 

By Lemma~\ref{lem:disjoint}, it is possible to go from $\reverse{C\setminus f}{\ori O_1}$ to $\ori O_2$ by reversing a disjoint collection of circuits and cocircuits. Suppose for the sake of contradiction that this involves swapping at least one circuit, and call this circuit $\vec{D}$.  

Take the 1-chain $\vec C + \vec D$ which is in $\ker(A)$ since both $\vec C$ and $\vec D$ are. Furthermore, 
$\ori O_1 \oriat {f} = \reverse{C\setminus f}{\ori O_1}\oriat {f}$. Since $f \in C$, $\vec {C} \compat \ori O_1$, and $\vec {D} \compat \reverse{C\setminus f}{\ori O_1}$, it follows that $f \in \supp(\vec C + \vec {D})$. 

Write $\vec C + \vec D$ as a sum of signed circuits as indicated by Lemma \ref{lem:decompose}. Take a signed circuit $\vec C'$ of the sum whose support contains $f$ (There is at least one such circuit). We claim that we must have $\vec {C'} \compat \ori O_1$. This follows from the fact that the support of $\vec {C'}$ must be a subset of the support of $\vec C + \vec {D}$ as well as the condition that $\vec C'\oriat x$ must have the same sign as $(\vec C + \vec {D})\oriat x$ for all $x \in C'$. In particular, $\vec C'$ is compatible with $\ori O_1$. 

We have therefore shown that $\vec{C'}$ is a signed circuit such that $\vec f \in \vec {C'}$ and $\vec C' \compat \ori O_1$. Next, we will show that $\flip(C') \subsetneq \flip (C)$, contradicting the minimality condition on $C$.

For $x \in C' \setminus C$, we must have $x \in D$. Since $D$ is disjoint from all other circuits and cocircuits reversed to go from $\reverse{C\setminus f}{\ori O_1}$ to $\ori O_2$, it follows that $\ori O_2 \oriat{x} \not= \reverse{C\setminus f}{\ori O_1}\oriat{x} = \ori O_1 \oriat{x}$ where in the last equality we use $x\notin C$. 
Furthermore, since $x \in C'$, we must have $\ori O_1 \oriat{x} \not= \reverse{C'\setminus f}{\ori O_1}\oriat{x}$. Combining these two inequalities implies that $\ori O_2 \oriat{x} = \reverse{C'\setminus f}{\ori O_1}\oriat{x}$ and $x \not\in \flip(C')$. 

Next, consider $x \in (C' \cap C)$. Here, it is immediate 
that $\reverse{C'\setminus f}{\ori O_1}\oriat{x} = \reverse{C\setminus f}{\ori O_1}\oriat{x}$. In particular, $x \in \flip(C')$ if and only if $x \in \flip (C)$ for this case. 

Because $C' = (C' \setminus C) \cup (C' \cap C)$, we have demonstrated that $\flip(C') \subseteq \flip (C)$. The last thing that we need to prove our claim is that $\flip(C') \not= \flip(C)$. By Lemma~\ref{lem:disjoint}, we know that $\vec {D} \compat \reverse{C\setminus f}{\ori O_1}$ but $\vec{D} \not\in \sigma$. If $(D \cap C)\setminus f = \emptyset$, then 
we have $D\compat \ori O_1$. However, since $\ori O_1 \in \ccmin(M)$, this contradicts the fact that $\vec{D} \not\in \sigma$. In particular, $(D \cap C)\setminus f \neq \emptyset$. Fix an element $g \in (D \cap C) \setminus f$. Since $\vec {D} \compat \reverse{C\setminus f} {\ori O_1}$, we must have $\vec{C} \oriat{g} = (-\vec D)\oriat{g}$. This means that $g \not\in C'$, and thus $g \not\in \flip(C')$. However, since $g$ is flipped twice to go from $\ori O_1$ to $\ori O_2$, we also have $g \in \flip(C)$. This completes the proof that $\flip(C')\subsetneq \flip(C)$, and shows by contradiction that no circuits are reversed when going from $\reverse{C\setminus f}{\ori O_1}$ to $\ori O_2$.

Finally, we show that at most one cocircuit reversal is required (and this reversal occurs if and only if $-\vec f \compat \ori O_2$). Consider any signed cocircuit $\vec {C^*}$ that is compatible with $\reverse{C\setminus f}{\ori O_1}$, but is not in $\sigma^*$. If $C^* \cap C = \emptyset$, then we also have $\vec C^* \compat \ori O_1$. 
However, this contradicts the condition that $\ori O_1$ is $(\sigma, \sigma^*)$-compatible. Thus, $C^* \cap C$ must be nonempty. By applying Lemma~\ref{lem:ori_intersection} to the orientation $\reverse{C \setminus f}\ori{O}_1$, the circuit $-\vec{C}$, and the cocircuit $\vec{C}^*$, we have $f \in C^*$. 

Let $\vec{C^*_1},\ldots,\vec{C^*_t}$ be the cocircuits that have to be reversed in $\reverse{C\setminus f}{\ori O_1}$ to get $\ori O_2$. We have shown that for every $i$, the signed cocircuit $\vec C^*_i$ must contain $f$. However, these signed cocircuits must all be disjoint, so there can be at most one of them. In other words, $t \le 1$. 

If $t = 0$, then we have $\ori O_1\oriat {f} = \reverse{C\setminus f}{\ori O_1}\oriat{f} = \ori O_2 \oriat{f}$ However, if $t = 1$, we have $\ori O_1\oriat {f} = \reverse{C\setminus f}{\ori O_1}\oriat{f} = -\ori O_2 \oriat{f}$. In particular, this final reversal occurs if and only if $- \vec f \compat \ori O_2$.

This proves $(a),(b)$ and $(c)$ for this case. Let us also prove $(d)$. 
We have already seen that we must have $f \in C \cap C^*$. Since $\vec C \compat \ori O_1$ and $\vec {C^*} \compat \ori O_2$, we have that $(-\vec C \setminus f)$ and $(-\vec C^* \setminus f)$ are both compatible with $\reverse{C} \ori O_1$. By Lemma~\ref{lem:ccint}, there must be some $g \in (C \cap C^*)\setminus f$, and this $g$ must be unique by Lemma~\ref{lem:ori_intersection}. The result $\ori O_2 = \reverse{(C\cup C^*)\setminus g}\ori{O}_1$ follows immediately.\\

\begin{enumerate}[resume* = subcases]
\item Lemma \ref{lem:swapf}(3) applies:\\
\end{enumerate}

This case is completely analogous to the previous one. The only difference is that we swap the role of circuits and cocircuits throughout. 

\end{proof}

\section{A Proof of Consistency} \label{sec:main}

In this section, we prove Theorem~\ref{thm:consistency}, which is the main result of the paper. Let us begin by restating the result. 

\setcounter{thmx}{0}
\begin{thmx}
    Let $M$ be a regular matroid and $(\sigma,\sigma^*)$ be a triangulating circuit-cocircuit signature. Suppose that $\vec f$ is an arc and $B_1,B_2 \in \mathbf B(M)$ such that
    \[[\vec f] \cdot \BBY_{(M,\sigma,\sigma^*)}(B_1)=\BBY_{(M,\sigma,\sigma^*)}(B_2).\] Then, the following 3 properties must hold. 
\begin{enumerate}
    \item For any $e \in (B_1^c \cap B_2^c)\setminus f$, we have
    \[ [\vec f] \cdot \BBY_{(M\setminus e,\sigma\setminus e,\sigma^*\setminus e)}(B_1)=\BBY_{(M\setminus e,\sigma\setminus e,\sigma^*\setminus e)}(B_2).\]
    \item For any $e \in (B_1 \cap B_2)\setminus f$, we have
    \[ [\vec f] \cdot \BBY_{(M/ e,\sigma/ e,\sigma^*/ e)}(B_1\setminus e)=\BBY_{(M/ e,\sigma/ e,\sigma^*/ e)}(B_2\setminus e).\]
    \item For any $e\in E$ that is in a different connected component of $M$ than $f$, we have 
    \[e \in B_1 \iff e \in B_2.\]
\end{enumerate}
\end{thmx}

Most of our effort will go towards proving Theorem~\ref{thm:consistency}(1), which first requires a few technical lemmas. After this, we prove part (2) of the theorem by duality and then part (3) by a simple argument.  

However, before proving Theorem~\ref{thm:consistency}(1), we will focus a subsection on proving a few technical lemmas, namely Lemmas~\ref{lem:oldcase} and~\ref{lem:newcase}.

\subsection{Proof of some Technical Lemmas}\label{sec:technical}

In this subsection, we take a close look at fourientations induced by $B_1$ and $B_2$. For better readability, let us introduce some simplified notations. For all equations, we have $i \in \{1,2\}$. 
\begin{equation}\label{eq:notations}
\fouri B_i := \fouri F(B_i,\sigma) \hspace{1.5 cm}\fouri B_i^* := \fouri F(B_i,\sigma^*) 
\end{equation}
\begin{equation}\label{eq:orientation_names}
\ori O_i:=\BBY_{(M,\sigma,\sigma^*)}(B_i)  =  \fouri B_i \cap \fouri B_i^*
\end{equation}


We also define a pair of fourientations, which will be used throughout the proof. 
\begin{equation}\label{eq:F_and_Fstar_def}
\begin{split}
\text{Let $\fouri{F}:=\fouri{B}_1 \cap -\fouri{B}_2$ and $\fouri{F^*}:=(\fouri{B}_1^* \cap -\fouri{B}_2^*)^c$.} 
\end{split}
\end{equation}

Throughout this subsection, we will assume that we have a fixed pair of bases and a fixed triangulating circuit-cocircuit signature so that~\eqref{eq:notations},~\eqref{eq:orientation_names}, and~\eqref{eq:F_and_Fstar_def} can all be applied. 

After unraveling what the notations mean, we obtain Table~\ref{table:F_andF*}, which shows how $\fouri{F}$ and $\fouri{F^*}$ relate to $\ori O_1$ and $\ori O_2$.

\begin{remark}\label{rem:intuition} 
Let us remark on the intuitive meanings of $\fouri F$ and $\fouri F^*$. The fourientations $\fouri B_1$ and $\fouri B_2$ encode information about $\ori O_1$ and $\ori O_2$ based on the circuit signature $\sigma$. It follows that $\fouri F$ encodes partial information about which edges switch when going from $\ori O_1$ to $\ori O_2$. More precisely, if $\fouri F\oriat{x}=\emptyset$, then we know that $x$ must reverse when going from $\ori O_1$ to $\ori O_2$. If $\fouri F\oriat{x}=+$ (resp. $\fouri F\oriat{x}=-$), then $x$ might or might not be reversed, but if it is reversed, then it must be true that $\ori O_1\oriat{x}=+$ (resp. $\ori O_1\oriat{x}=-$). The situation is analogous for the fourientation $-(\fouri F^*)^c$. The difference is that this fourientation encodes reversal information based on the cocircuit signature $\sigma^*$. 
\end{remark}


\begin{table}[h!]
\centering
\bgroup
\def\arraystretch{2}
\begin{tabular}{ |c||c|c|c|c|c|c| }

\hline
 & $\fouri{B}_1\oriat{x}$ & $\fouri{B}^*_1\oriat{x}$ & $\fouri{B}_2\oriat{x}$ & $\fouri{B}^*_2\oriat{x}$ & $\fouri F\oriat{x}$ & $\fouri F^*\oriat{x}$\\
\hline
\hline
$x \in B_1 \cap B_2$   & $\pm$ & $\ori{O}_1\oriat{x}$   & $ \pm$ & $\ori{O}_2\oriat{x}$ & $\pm$ & $(-\ori O_1\oriat x) \cup \ori O_2\oriat x$ \\
\hline
$x \in B_1 \setminus B_2$ &$\pm$  & $\ori{O}_1\oriat{x}$  &  $\ori{O}_2\oriat{x}$ & $ \pm$    & $-\ori O_2\oriat x$ & $-\ori O_1\oriat x$\\
\hline
$x \in B_2 \setminus B_1$ & $\ori{O}_1\oriat{x}$ & $\pm$    & $\pm$ & $\ori{O}_2\oriat{x}$ & $\ori O_1\oriat x$ & $\ori O_2\oriat x$\\
\hline
$x \in B_1^c \cap B_2^c$ &$\ori{O}_1\oriat{x}$  & $\pm$  &  $\ori{O}_2\oriat{x}$ & $ \pm$    & $\ori O_1\oriat x \cap (-\ori O_2\oriat{x})$ & $\emptyset$\\
\hline
\end{tabular}
\egroup
\caption{For any $x \in E$, this chart shows the value of $\fouri{B}_1\oriat{x}$, $\fouri{B}^*_1\oriat{x}$, $\fouri{B}_2\oriat{x}$, $\fouri{B}^*_2\oriat{x}$, $\fouri F\oriat{x}$, and $\fouri F^*\oriat{x}$ in terms of $\ori O_1\oriat{x}$ and $\ori O_2\oriat{x}$. }\label{table:F_andF*}

\end{table}


Some useful implications of Table~\ref{table:F_andF*} are given by the following lemma. Note that properties (1) and (3) are also proven in \cite[Lemma 2.8]{Ding2}.

\begin{lemma}\label{lem:info in the table}
For $x \in E$, the following properties hold. 
\begin{enumerate}
    \item If $\ori{ O}_1\oriat{x}=\ori{ O}_2\oriat{x}$, then $\fouri{F}\oriat{x}=\fouri{F^*}\oriat{x}$. 
    \item If $\ori{ O}_1\oriat{x}=\ori{ O}_2\oriat{x}$ and $x\in B_1^c \cap B_2^c$, then $\fouri{F}\oriat{x}=\fouri{F^*}\oriat{x}=\emptyset$. 
    
    \item If $\ori{ O}_1\oriat{x}\neq\ori{ O}_2\oriat{x}$, then $\ori{O}_1\oriat{x}\subseteq\fouri{F}\oriat{x}$ and $\fouri{F^*}\oriat{x}\subseteq\ori{O}_2\oriat{x}$.
     
    \item  If $\ori{ O}_1\oriat{x}\neq\ori{ O}_2\oriat{x}$ and $x\in B_1^c \cap B_2^c$, then $\fouri{F}\oriat{x}=\ori{O}_1\oriat{x}$ and $\fouri{F^*}\oriat{x}=\emptyset$. 
\end{enumerate}
\end{lemma}

Our next lemma is immediate from Definition~\ref{def:triangulating}, but is important enough to be worth restating. 

\begin{lemma}\label{lem:keylemma} Using the notations of \eqref{eq:notations} and \eqref{eq:F_and_Fstar_def}, the following properties must hold. 
    \begin{enumerate}
        \item The fourientation $\fouri F$ is not compatible with any $\vec C \in \SignedCircuits M$. 
        \item The fourientation $-(\fouri F^*)^c$ is not compatible with any $\vec {C^*} \in \SignedCoCircuits M$. 
    \end{enumerate} 
\end{lemma} 

Now, we are ready to prove Lemma~\ref{lem:oldcase}, the first of the technical lemmas that are necessary for the proof of Theorem~\ref{thm:consistency}(1).

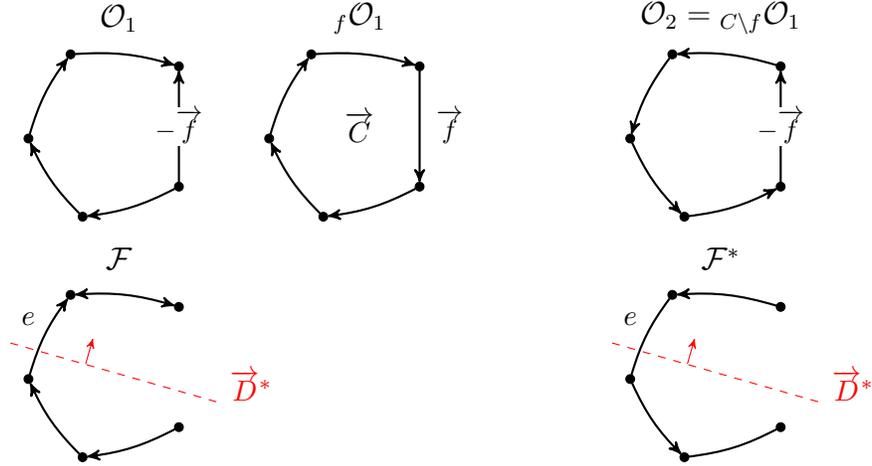
\begin{figure}[t]
            \centering

    \begin{tikzpicture}[scale=0.8]
    
    \tikzstyle{o}=[circle,fill,scale=.3,draw]
    \begin{scope}[shift={(-5,0)}]
	\node at (1,2.8) {$\ori O_1$};
	\node [o] (5) at (-0.5,0.8) {};
    \node [o] (1) at (0.4,-0.5) {};
	\node [o] (2) at (2,0) {};
	\node [o] (3) at (2,2) {};
	\node [o] (4) at (0.2,2.2) {};
    \draw [thick,->,>=stealth',bend left=10] (2) to (1);
	\draw [thick,->,>=stealth',bend left=10] (1) to (5);
	\draw [thick,->,>=stealth',bend left=10] (5) to (4);
	\draw [thick,->,>=stealth',bend left=10] (4) to (3);
    \draw [thick,->,>=stealth'] (2) to node[fill = white,inner sep=1pt]{\footnotesize $-\vec f$} (3);
    \end{scope}

    \begin{scope}[shift={(-1,0)}]
	\node at (1,2.8) {$\reverse{f} \ori O_1$};
	\node [o] (5) at (-0.5,0.8) {};
    \node [o] (1) at (0.4,-0.5) {};
	\node [o] (2) at (2,0) {};
	\node [o] (3) at (2,2) {};
	\node [o] (4) at (0.2,2.2) {};
    \draw [thick,->,>=stealth',bend left=10] (2) to (1);
	\draw [thick,->,>=stealth',bend left=10] (1) to (5);
	\draw [thick,->,>=stealth',bend left=10] (5) to (4);
	\draw [thick,->,>=stealth',bend left=10] (4) to (3);
	\draw [thick,<-,>=stealth'] (2) to (3);
    \node[] at (2.5, 1) {\footnotesize $\vec f$};
    \node[] at (1, 1) {\small $\vec C$};
    \end{scope}
    
    \begin{scope}[shift={(5,0)}]
	\node at (1,2.8) {$\ori O_2 = \reverse{C\setminus f} \ori O_1$};
    \node [o] (5) at (-0.5,0.8) {};
    \node [o] (1) at (0.4,-0.5) {};
	\node [o] (2) at (2,0) {};
	\node [o] (3) at (2,2) {};
	\node [o] (4) at (0.2,2.2) {};
    \draw [thick,<-,>=stealth',bend left=10] (2) to (1);
	\draw [thick,<-,>=stealth',bend left=10] (1) to (5);
	\draw [thick,<-,>=stealth',bend left=10] (5) to (4);
	\draw [thick,<-,>=stealth',bend left=10] (4) to (3);
	\draw [thick,->,>=stealth'] (2) to node[fill = white,inner sep=1pt]{\footnotesize $-\vec f$} (3);
    \end{scope}

    \tikzstyle{o}=[circle,fill,scale=.3,draw]
    \begin{scope}[shift={(-5,-4)}]
	\node at (1,2.8) {$\fouri F$};
	\node [o] (5) at (-0.5,0.8) {};
    \node [o] (1) at (0.4,-0.5) {};
	\node [o] (2) at (2,0) {};
	\node [o] (3) at (2,2) {};
	\node [o] (4) at (0.2,2.2) {};
    \draw [thick,->,>=stealth',bend left=10] (2) to (1);
	\draw [thick,->,>=stealth',bend left=10] (1) to (5);
	\draw [thick,->,>=stealth',bend left=10] (5) to (4);
	\draw [thick,<->,>=stealth',bend left=10] (4) to (3);
    \node at (-0.5, 1.8) {\small $e$};
    \draw [-,dashed,color=red] (-0.8, 1.4) to (2.7, 0.4);
	\draw [->,>=stealth',color=red] (0.45, 1.05) to (0.57, 1.47);
    \node [color=red] at (3.2, 0.7) {\small $\vec D^*$};
    \end{scope}

    \tikzstyle{o}=[circle,fill,scale=.3,draw]
    \begin{scope}[shift={(5,-4)}]
	\node at (1,2.8) {$\fouri F^*$};
	\node [o] (5) at (-0.5,0.8) {};
    \node [o] (1) at (0.4,-0.5) {};
	\node [o] (2) at (2,0) {};
	\node [o] (3) at (2,2) {};
	\node [o] (4) at (0.2,2.2) {};
    \draw [thick,-,>=stealth',bend left=10] (2) to (1);
	\draw [thick,<-,>=stealth',bend left=10] (1) to (5);
	\draw [thick,-,>=stealth',bend left=10] (5) to (4);
	\draw [thick,<-,>=stealth',bend left=10] (4) to (3);
    \node at (-0.5, 1.8) {\small $e$};
    \draw [-,dashed,color=red] (-0.8, 1.4) to (2.7, 0.4);
	\draw [->,>=stealth',color=red] (0.45, 1.05) to (0.57, 1.47);
    \node [color=red] at (3.2, 0.7) {\small $\vec D^*$};
    \end{scope}
\end{tikzpicture}

            \caption{A pair of orientations and a pair of fourientations corresponding to the setup of Lemma \ref{lem:oldcase}. 
            The first row shows how $\ori O_1$ is transformed into $\ori O_2$. 
            In the second row, for the edges not drawn in $\fouri F$ and $\fouri F^*$, the two fourientations coincide; for the edges drawn, the two pictures show the possible ways how $\fouri F$ and $\fouri F^*$ can differ in the sense of Lemma~\ref{lem:info in the table}(3).
            }
            \label{fig:easycase}
\end{figure}

\begin{lemma}\label{lem:oldcase}
Consider the setup of Theorem~\ref{thm:consistency}(1), and use the notations of \eqref{eq:orientation_names}.
Further assume that $\ori O_2$ is obtained from $\ori O_1$ by reversing $f$ and then reversing a signed circuit $\vec C$ containing $f$. Then,  \[(B_1^c \cap B_2^c \cap C)\setminus f = \emptyset.\]
\end{lemma}
\begin{proof}


We use the notations \eqref{eq:notations} and \eqref{eq:F_and_Fstar_def}.

First, note that since $f$ is reversed twice, and the other elements of $C$ are reversed once, we have $\ori O_1 =  \reverse{C \setminus f}{\ori O_2}$. This is also shown in the first row of Figure~\ref{fig:easycase}. 

By Lemma~\ref{lem:info in the table}(3), for $x\in C\setminus f$, we have $\ori{O}_1\oriat{x}\subseteq\fouri{F}\oriat{x}$. Hence, $\vec C\setminus f$ is compatible with $\fouri{F}$. This is also demonstrated in the the second row of Figure~\ref{fig:easycase}.

Suppose for the sake of contradiction that there exists some $e \in (B_1^c \cap B_2^c \cap C)\setminus f$. Then, since $e \in C \setminus f$, we know that $\ori {O}_1 \oriat{e} \not= \ori{O}_2 \oriat{e}$. Thus, we can apply Lemma~\ref{lem:info in the table}(4) and conclude that $\fouri{F}\oriat{e}=\ori{O}_1\oriat{e}$ and $\fouri{F^*}\oriat{e}=\emptyset$.

Since $\fouri{F}\oriat{e}=\ori{O}_1\oriat{e}$, we know that $\fouri{F}\oriat{e} \in \{-,+\}$. Furthermore, by Lemma~\ref{lem:keylemma}, we know that $\fouri{F}$ cannot be compatible with any signed circuits. In particular, the conditions for Lemma~\ref{lem:3painting} apply, but the first possibility is impossible. This means that there exists a signed cocircuit $\vec{D^*}$ such that $e\in D^*$ and $\vec{D^*}$ is compatible with 
$(-\fouri F)^c$. 

Because $\vec C\setminus f$ is compatible with $\fouri{F}$, it follows from Lemma~\ref{lem:intersection} that $D^*\cap C=\{e, f\}$. This idea is also demonstrated in the second row in Figure~\ref{fig:easycase}. 

Next, we claim that $\vec{D^*}$ is also compatible with 
$(-\fouri F^*)^c$. We can show this directly by considering elements of $D^*$ that are in $C\setminus f$, and elements of $D^*$ that are not in $C \setminus f$. 
\begin{enumerate}
    \item We proved above that $e$ is the only element of $D^*$ that is also in $C\setminus f$. For this case, recall that $\fouri{F^*}\oriat{e}=\emptyset$. This means that $(-\fouri{F^*})^c\oriat{e}=\pm$ and the compatibility must hold regardless of $\vec{D^*}\oriat{e}$. 
    
    \item For element $x$ that is in $D^*$ but not in $C\setminus f$, we can use that fact that $\ori O_1 =\reverse{C \setminus f}{\ori O_2}$ to conclude that $\ori O_1\oriat{x} = \ori O_2 \oriat{x}$. In particular, we can apply Lemma~\ref{lem:info in the table}(1) to say that $\fouri{F}\oriat{x} = \fouri{F^*}\oriat{x}$. Because $\vec{D^*}$ is compatible with $(-\fouri F)^c$ by construction, it follows that $\vec{D^*}$ is also compatible with $(-\fouri F^*)^c$.
\end{enumerate}

We have now shown that there exists a signed cocircuit $\vec{D^*}$ that is compatible with $(-\fouri F^*)^c$. However, this directly contradicts the conclusion of Lemma~\ref{lem:keylemma}(2). 
\end{proof}

In addition to Lemma~\ref{lem:oldcase}, we will need one more lemma for the proof of Theorem~\ref{thm:consistency}(1), namely Lemma~\ref{lem:newcase}. First, we introduce an auxiliary lemma that gives compatibility conditions for the sum of two simple $1$-chains.

\begin{lemma}\label{lem:sumcompat} Let $\vec P$ and $\vec Q$ be simple 1-chains and let $\fouri F$ be a fourientation. The 1-chain $\vec P + \vec Q$ is compatible with $\fouri F$ if and only if for every $x \in P \cup Q$, at least one of the following properties hold: 
\renewcommand{\theenumi}{\roman{enumi}}
\begin{enumerate}
    \item $\fouri F\oriat{x} = \pm$
    \item $\vec P\oriat {x} \compat \fouri F\oriat{x}$
    \item $\vec Q\oriat {x} \compat \fouri F\oriat{x}$
    \item $\vec P\oriat {x} = -\vec Q\oriat {x}$
\end{enumerate}
\end{lemma}
\begin{proof}By definition, $\vec P + \vec Q$ is compatible with $\fouri F$ if and only if for every $x \in E$, we have $(\vec P + \vec Q)\oriat{x} = 0$, $\fouri F\oriat{x} = \pm$, or $\fouri F \oriat {x}$ matches the sign of $(\vec P + \vec Q)\oriat{x}$. If $x \not\in (P\cup Q)$ or if $\vec P\oriat {x} = -\vec Q\oriat {x}$, then this implies that $(\vec P + \vec Q)\oriat{x} = 0$. Otherwise, if $\fouri F\oriat{x} \not= \pm$, then we must have $\fouri F\oriat{x} \in \{-,+\}$. It follows quickly from the fact that $\vec P$ and $\vec Q$ are simple that $(\vec P +\vec Q)\oriat {x} \compat \fouri F\oriat{x}$ if and only if $\vec P\oriat {x} \compat \fouri F\oriat{x}$ or $\vec Q\oriat {x} \compat \fouri F\oriat{x}$. 
\end{proof}

\begin{remark}Note that whenever condition (i) of Lemma~\ref{lem:sumcompat} is satisfied, either condition (ii) or (iii) must also be satisfied. In particular, this means that Lemma~\ref{lem:sumcompat} would still be valid with this condition removed. We include it so that the lemma is slightly easier to apply. \end{remark} 



\begin{lemma}\label{lem:newcase}
Consider the setup of Theorem \ref{thm:consistency}(1), and use the notations of \eqref{eq:orientation_names}.
Further assume that $\ori O_2$ is obtained from $\ori O_1$ by reversing a signed circuit $\vec C$ containing $f$, then reversing $f$, and then finally reversing a signed cocircuit $\vec{C^*}$ containing $f$. Then  \[(B_1^c \cap B_2^c \cap C)\setminus f = \emptyset.\]
\end{lemma}
\begin{proof}
The conditions of the lemma imply that $\vec C \compat \ori{O}_1$ and $-\vec{C^*} \compat \ori{O}_2$.
Moreover, by Theorem~\ref{thm:generalDescription}(d), we know that $C\cap C^*=\{f,g\}$ for some $g\neq f$. Furthermore, this theorem also says that $\ori {O}_1 = \reverse {P}{\ori{O}_2}$, where $P:=(C\cup C^*)\setminus g$. In other words, $P$ is precisely the subset of $E$ where $\ori O_1$ and $\ori O_2$ differ. This is also shown in the first row of Figure~\ref{fix}.




For $x \in E\setminus P$, we can apply Lemma~\ref{lem:info in the table}(1) to conclude that $\fouri{F}\oriat{x} = \fouri{F^*}\oriat{x}$. Furthermore, for $x \in P$, we can apply Lemma~\ref{lem:info in the table}(3) to conclude that $\ori O_1 \oriat{x} \subseteq \fouri F\oriat{x}$ and $\fouri{F^*} \oriat{x} \subseteq \ori O_2\oriat{x}$. This second containment also implies that $\ori O_2\oriat{x}\subseteq (-\fouri{F^*})^c \oriat{x}$. This is also shown in the second row of Figure~\ref{fix}.

Next we claim that $\vec C\setminus g$ is compatible with $\fouri{F}$ and $-\vec{C^*}\setminus g$ is compatible with  $(-\fouri{F^*})^c$. Both of these claims follow from the results from the previous two paragraphs. In particular, $(\vec C\setminus g) \subset \vec C$, so we know that $(\vec C\setminus g) \compat \ori{O}_1$. Furthermore, it follows from $\ori O_1 \oriat{x} \subseteq \fouri F\oriat{x}$ that $(\vec C\setminus g) \compat \fouri{F}$. The argument that $(-\vec{C^*}\setminus g) \compat (-\fouri{F^*})^c$ is analogous.

 
\begin{figure}[h]
            \centering

    \begin{tikzpicture}[scale=1]
    \tikzstyle{o}=[circle,fill,scale=.3,draw]
    \begin{scope}[shift={(-5,0)}]
	\node at (0.5,2.8) {$\ori O_1$};
	\node [o] (0) at (0,0.2) {};
    \node [o] (1) at (1.5,-0.6) {};
	\node [o] (2) at (3,0.2) {};
	\node [o] (3) at (3,1.8) {};
	\node [o] (4) at (1.5, 2.8) {};
    \node [o] (5) at (0, 1.8) {};
    \draw [color=red,dashed,->,>=stealth',bend right=40] (3.8, 1) to (2,3.3);
    \node[] at (4.1, 2) {\color{red} $\vec{C}$};
    \node[] at (-0.35, 1) {\small $g$};
    \draw [thick,->,>=stealth',bend right=10] (0) to (1);
	\draw [thick,->,>=stealth',bend right=10] (1) to (2);
	\draw [thick,->,>=stealth',bend right=10] (2) to node[fill = white,inner sep=1pt]{\footnotesize $\vec f$} (3);
	\draw [thick,->,>=stealth',bend right=10] (3) to (4);
	\draw [thick,->,>=stealth',bend right=10] (4) to (5);
    \draw [thick,->,>=stealth',bend right=10] (5) to (0);
    \node [o] (6) at (0.75,0.2) {};
    \node [o] (7) at (1.5,0.2) {};
	\node [o] (8) at (2.25,0.2) {};
	\node [o] (9) at (0.75,1.8) {};
    \node [o] (10) at (1.5, 1.8) {};
	\node [o] (11) at (2.25, 1.8) {};
	\draw [thick,->,>=stealth'] (6) to (9);
	\draw [thick,->,>=stealth'] (7) to (10);
	\draw [thick,->,>=stealth'] (8) to (11);
    \end{scope}

    \begin{scope}[shift={(5,0)}]
	\node at (0.5,2.8) {$\ori O_2$};
	\node [o] (0) at (0,0.2) {};
    \node [o] (1) at (1.5,-0.6) {};
	\node [o] (2) at (3,0.2) {};
	\node [o] (3) at (3,1.8) {};
	\node [o] (4) at (1.5, 2.8) {};
    \node [o] (5) at (0, 1.8) {};
    \draw [color=red,dashed,-] (-0.5, 1.5) to (4,1.5);
    \draw [->,>=stealth',color=red] (3.8, 1.5) to (3.8, 1.25);
    \node[] at (-0.9, 1.7) {\color{red} $-\vec{C^*}$};
    \draw [thick,<-,>=stealth',bend right=10] (0) to (1);
	\draw [thick,<-,>=stealth',bend right=10] (1) to (2);
	\draw [thick,<-,>=stealth',bend right=10] (2) to node[fill = white,inner sep=1pt]{\footnotesize $-\vec f$} (3);
	\draw [thick,<-,>=stealth',bend right=10] (3) to (4);
	\draw [thick,<-,>=stealth',bend right=10] (4) to (5);
	\draw [thick,->,>=stealth',bend right=10] (5) to (0);
    \node[] at (-0.35, 1) {\small $g$};
    \node [o] (6) at (0.75,0.2) {};
    \node [o] (7) at (1.5,0.2) {};
	\node [o] (8) at (2.25,0.2) {};
	\node [o] (9) at (0.75,1.8) {};
    \node [o] (10) at (1.5, 1.8) {};
	\node [o] (11) at (2.25, 1.8) {};
	\draw [thick,<-,>=stealth'] (6) to (9);
	\draw [thick,<-,>=stealth'] (7) to (10);
	\draw [thick,<-,>=stealth'] (8) to (11);
    \end{scope}

    \tikzstyle{o}=[circle,fill,scale=.3,draw]
    \begin{scope}[shift={(-5,-4)}]
		\node at (0.5,2.8) {$\fouri F$};
	\node [o] (0) at (0,0.2) {};
    \node [o] (1) at (1.5,-0.6) {};
	\node [o] (2) at (3,0.2) {};
	\node [o] (3) at (3,1.8) {};
	\node [o] (4) at (1.5, 2.8) {};
    \node [o] (5) at (0, 1.8) {};
    \draw [thick,->,>=stealth',bend right=10] (0) to (1);
	\draw [thick,<->,>=stealth',bend right=10] (1) to (2);
	\draw [thick,->,>=stealth',bend right=10] (2) to node[fill = white,inner sep=1pt]{\footnotesize $\vec f$} (3);
	\draw [thick,->,>=stealth',bend right=10] (3) to (4);
	\draw [thick,->,>=stealth',bend right=10] (4) to (5);
    \node[] at (2.8, 2.3) {\footnotesize $e$};
    \node [o] (6) at (0.75,0.2) {};
    \node [o] (7) at (1.5,0.2) {};
	\node [o] (8) at (2.25,0.2) {};
	\node [o] (9) at (0.75,1.8) {};
    \node [o] (10) at (1.5, 1.8) {};
	\node [o] (11) at (2.25, 1.8) {};
	\draw [thick,->,>=stealth'] (6) to (9);
	\draw [thick,->,>=stealth'] (7) to (10);
	\draw [thick,<->,>=stealth'] (8) to (11);
    \draw [-,color=red,dashed,rounded corners=5pt] (-0.5, 0.7) -- (1, 1.4) --  (1.2,2.1) -- (3.1,2.8);
	\draw [->,>=stealth',color=red] (3, 2.75) to (2.92, 2.98);
    \node[] at (-0.9, 0.7) {\color{red} $\vec{D^*}$};
    \end{scope}

    \tikzstyle{o}=[circle,fill,scale=.3,draw]
    \begin{scope}[shift={(5,-4)}]
	\node at (0.5,2.8) {$\fouri F^*$};
	\node [o] (0) at (0,0.2) {};
    \node [o] (1) at (1.5,-0.6) {};
	\node [o] (2) at (3,0.2) {};
	\node [o] (3) at (3,1.8) {};
	\node [o] (4) at (1.5, 2.8) {};
    \node [o] (5) at (0, 1.8) {};
    \draw [thick,<-,>=stealth',bend right=10] (0) to (1);
	\draw [thick,<-,>=stealth',bend right=10] (1) to (2);
	\draw [thick,<-,>=stealth',bend right=10] (2) to node[fill = white,inner sep=1pt]{\footnotesize $-\vec f$} (3);
	\draw [thick,-,>=stealth',bend right=10] (3) to (4);
	\draw [thick,-,>=stealth',bend right=10] (4) to (5);
    \node[] at (2.8, 2.3) {\footnotesize $e$};
    \node [o] (6) at (0.75,0.2) {};
    \node [o] (7) at (1.5,0.2) {};
	\node [o] (8) at (2.25,0.2) {};
	\node [o] (9) at (0.75,1.8) {};
    \node [o] (10) at (1.5, 1.8) {};
	\node [o] (11) at (2.25, 1.8) {};
	\draw [thick,<-,>=stealth'] (6) to (9);
	\draw [thick,-,>=stealth'] (7) to (10);
	\draw [thick,<-,>=stealth'] (8) to (11);
    \draw [-,color=red,dashed,rounded corners=5pt] (-0.5, 0.7) -- (1, 1.4) --  (1.2,2.1) -- (3.1,2.8);
	\draw [->,>=stealth',color=red] (3, 2.75) to (2.92, 2.98);
    \node[] at (-0.9, 0.7) {\color{red} $\vec{D^*}$};
    \end{scope}
\end{tikzpicture}
            
            \caption{A pair of orientations and a pair of fourientations corresponding to the setup of Lemma \ref{lem:newcase}. 
            The first row shows how $\ori O_1$ is transformed into $\ori O_2$. In the second row, for the edges not drawn in $\fouri F$ and $\fouri F^*$, the two fourientations coincide; for the edges drawn, the two pictures show the possible ways how $\fouri F$ and $\fouri F^*$ can differ in the sense of Lemma~\ref{lem:info in the table}(3).}
            \label{fix}
\end{figure}
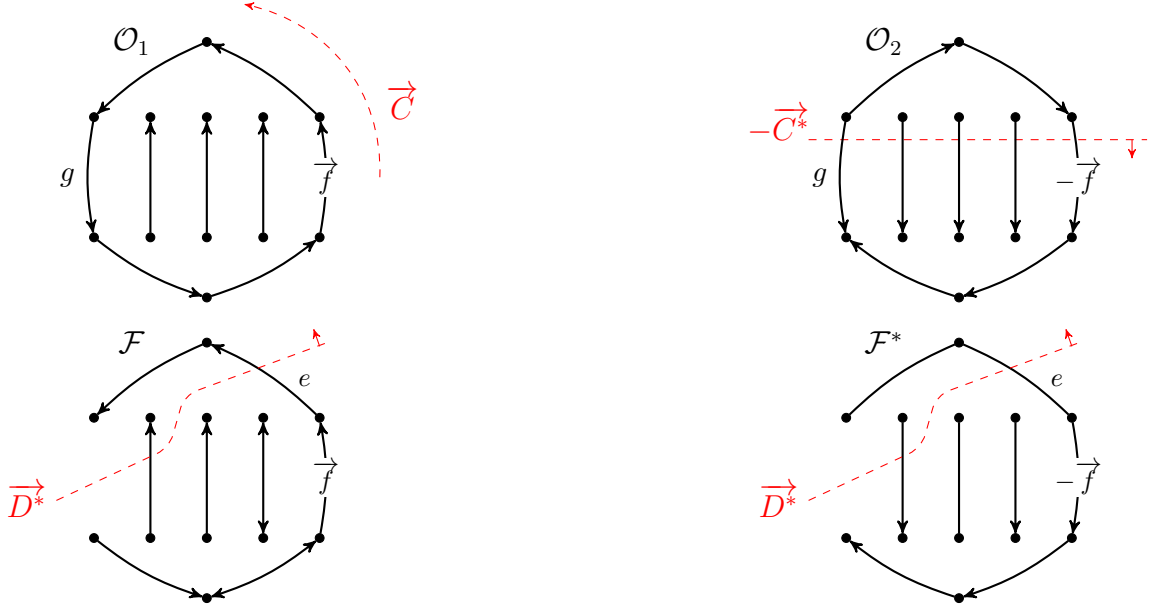


Assume for the sake of contradiction that there exists $e\in (B_1^c \cap B_2^c \cap C)\setminus f$. We will start by showing that $e\neq g$. Notice that we must have $\vec C\oriat{f} = \vec {C^*}\oriat{f}$ since $\vec C$ is compatible with $\ori O_1$ and $\vec {C^*}$ is compatible with $\reverse{C\setminus f}\ori O_1$. From here, it follows from Lemma~\ref{lem:ccint} that $\vec{C}\oriat{g} = -\vec{C^*}\oriat{g}$. If $e = g$, then we can apply Lemma~\ref{lem:info in the table}(2) to say that $\fouri{F^*}\oriat{g}=\emptyset$. Then, since $-\vec{C^*}\setminus g \compat (-\fouri{F^*})^c$, it follows that $-\vec{C^*} \compat (-\fouri{F^*})^c$. However, this contradicts Lemma~\ref{lem:keylemma}(2). Thus, we must have $e \not= g$. 

Since $e \in C$ and $e \not= g$, we know that $\ori O_1\oriat{e} \not= \ori O_2 \oriat{e}$. Thus, we can apply Lemma~\ref{lem:info in the table}(4) to conclude that $\fouri{F}\oriat{e}=\ori{O}_1\oriat{e}$ and $\fouri{F^*}\oriat{e}=\emptyset$. Because $\fouri{F}\oriat{e}=\ori{O}_1\oriat{e}$, it follows from Lemma~\ref{lem:3painting} that $e$ is contained in a signed circuit compatible with $\fouri F$ or a signed cocircuit compatible with $(-\fouri F)^c$. By Lemma~\ref{lem:keylemma}, the first option is impossible, so there must be a signed cocircuit containing $e$ that is compatible with $(-\fouri F)^c$. Call this signed cocircuit $\vec{D^*}$. 

Now we have a signed circuit $\vec{C}$ and a signed cocircuit $\vec{D^*}$ such that $e \in C \cap D^*$, $(\vec C \setminus g) \compat \fouri F$, and $\vec {D^*} \compat (-\fouri F)^c$. It follows from Lemma~\ref{lem:intersection} that $C \cap D^* = \{e,g\}$. Recall that $-\vec{C^*}$ is a signed cocircuit that is compatible with $\ori O_2$ and consider the 1-chain $-\vec{C^*} + \vec{D^*}$. Note that the support of this one chain is nonempty because $e \in D^* \setminus C^*$. 

Our goal for the rest of the proof will be to show that $-\vec{C^*} + \vec{D^*}$ is compatible with $(-\fouri{F^*})^c$. Then, since $-\vec{C^*} + \vec{D^*}\in\row(A)$, the row space of the fixed matrix realizing $M$, we can apply Lemma~\ref{lem:decompose} to write $-\vec{C^*} + \vec{D^*}$ as a sum of signed cocircuits that are compatible with $(-\fouri{F^*})^c$. This contradicts Lemma~\ref{lem:keylemma}(2) and concludes the proof. 
 


 
We apply Lemma~\ref{lem:sumcompat} to prove that $-\vec{C^*} + \vec{D^*}$ is compatible with $(-\fouri{F^*})^c$. For every $x \in C^* \cup D^*$, we consider the following four exclusive cases.

\begin{enumerate}[resume*=cases,start = 1]
\item $x = e$:\\ 
\end{enumerate}

We showed earlier (specifically in paragraph 5 of the current proof) that $\fouri{F^*}\oriat{x}=\emptyset$. Thus, $(-\fouri{F^*}^c)\oriat{x}=\pm$ and we satisfy condition (i) of Lemma~\ref{lem:sumcompat}. \\

\begin{enumerate}[resume*=cases]
\item $x = g$:\\ 
\end{enumerate}




Recall that $\vec {C} \setminus g$ is compatible with $\fouri F$ and that $\vec{D^*}$ is compatible with $(-\fouri F)^c$. Thus, we have $\vec {D^*}\oriat{e}=(-\fouri F)^c\oriat{e}=\fouri F\oriat{e}=\vec{C}\oriat{e}$. Since $D^* \cap C = \{e,g\}$, we can apply Lemma~\ref{lem:ccint} to conclude that $\vec{D^*}\oriat{g} = -\vec{C}\oriat{g}$. Furthermore, we showed that $ \vec{C}\oriat{g} = -\vec{C^*}\oriat{g}$. Thus, we have $\vec{D^*}\oriat{g} = \vec{C^*}\oriat{g}$, and condition (iv) of Lemma~\ref{lem:sumcompat} is satisfied.\\

\begin{enumerate}[resume*=cases]
\item $x\in C^*\setminus g$:\\ 
\end{enumerate}

We have already shown that $-\vec{C^*}\setminus g$ is compatible with  $(-\fouri{F^*})^c$. In particular, this means that $-\vec{C^*}\oriat{x}\compat (-\fouri{F^*})^c\oriat{x}$. Thus, condition (ii) of Lemma~\ref{lem:sumcompat} is satisfied.\\

\begin{enumerate}[resume*=cases]
\item $x\in D^*\setminus (C^*\cup e)$:\\ 
\end{enumerate}

Recall that $\ori {O}_1 = \reverse{P}{\ori {O}_2}$, where $P = (C \cup C^*) \setminus f$. Since $D^* \cap C = \{e,g\}$, we know that $x \not\in P$. In particular, this means that $\ori O_1\oriat{x} = \ori O_2\oriat{x}$. By Lemma~\ref{lem:info in the table}(1), this implies that $\fouri F\oriat{x} = \fouri {F^*} \oriat{x}$, which in turn implies that $(-\fouri F)^c\oriat{x} = (-\fouri {F^*})^c \oriat{x}$. By construction, we know that $\vec D^* \compat (-\fouri F)^c$. Thus, we must also have $\vec D^*\oriat{x} \compat (-\fouri F^*)^c\oriat{x}$ and condition (iii) of Lemma~\ref{lem:sumcompat} is satisfied.

\end{proof}
\subsection{Proving the Main Theorem}\label{sec:mainproof}

\begin{proof}[Proof of Theorem~\ref{thm:consistency}(1).]
For this proof, we will not need to use the notation from equations~\eqref{eq:notations} or~\eqref{eq:F_and_Fstar_def}, but we will use the following extension of~\eqref{eq:orientation_names}. 
\begin{equation}\label{eq:orientation_names2}
\text{for $i\in\{1,2\}$, }\hspace{ 1 cm}\ori O_i:=\BBY_{(M,\sigma,\sigma^*)}(B_i) \hspace{ 1 cm} \text{ and} \hspace{ 1 cm}\ori O'_i:=\BBY_{(M\setminus e,\sigma\setminus e,\sigma^*\setminus e)}(B_i).
\end{equation}
To reiterate our goal using this condensed notation, we are given that $[\vec{f}] \cdot \ori O_1 = \ori O_2$ and we need to show that $[\vec{f}] \cdot \ori O'_1 = \ori O'_2$. 

By Lemma~\ref{lem:BBY_delcont}, $\ori O_i$ and $\ori O'_i$ coincide on $E\setminus e$. We will show that the action of $[\vec{f}]$ on $\ori O_1$ descends to the action of $[\vec{f}]$ on $\ori O'_1$ as deleting $e$ does not affect the action in an essential way.

By applying Theorem~\ref{thm:generalDescription}, we can break up the problem into a manageable number of cases. First, we can consider whether or not $\vec f \compat \ori O_1$ and whether or not $\vec f \compat \ori O_2$.\\

\begin{enumerate}[resume*=cases,start = 1]
\item $\vec f \not\compat \ori O_1$ and $\vec f \compat \ori O_2$:\\ 
\end{enumerate}

It follows from Theorem~\ref{thm:generalDescription} that $\ori O_2 = \reverse {f}{\ori O_1}$. Since $\ori O_i$ and $\ori O'_i$ coincide on $E\setminus e$ for $i\in\{1,2\}$, we have $\ori O'_2 = \reverse {f}{\ori O'_1}$, which implies that $[\vec{f}] \cdot \ori O'_1 = \ori O'_2$.
\\
\begin{enumerate}[resume*=cases]
\item $\vec f \not\compat \ori O_1$ and $\vec f \not\compat \ori O_2$:\\
\end{enumerate}

By Theorem~\ref{thm:generalDescription}, to get $\ori O_2$ from $\ori O_1$, we reverse $\vec f$ first and then reverse exactly one signed circuit $\vec C$ or exactly one signed cocircuit $\vec{C^*}$. Furthermore, the circuit or cocircuit that we reverse must contain $f$. We consider the two possibilities in two separate subcases.  \\

\begin{enumerate}[resume* = subcases, start = 1] 
\item A signed circuit $\vec C$ is reversed:\\
\end{enumerate}

For this case, we first apply Lemma~\ref{lem:oldcase}, which implies that $e \in (B_1^c \cap B_2^c)\setminus f$. In particular, we have $e\notin C$. It follows that $\vec C \compat \reverse{f}\ori O'_1$, and that reversing $\vec f$ in $\ori O'_1$ and then $\vec C$ leads to $\ori O'_2$. In particular, we have $[\vec{f}] \cdot \ori O'_1 = \ori O'_2$ as desired. \\

\begin{enumerate}[resume* = subcases]
\item A signed cocircuit $\vec{C^*}$ is reversed:\\
\end{enumerate}

By Lemma~\ref{lem:breaking cocircuit}, $\vec{C^*}\setminus e$ is a signed cocircuit of $M\setminus e$ or a disjoint union of signed cocircuits of $M\setminus e$. It follows that $\ori O'_2$ can be obtained from $\ori O'_1$ by reversing $\vec f$ and then reversing the cocircuits that make up $\vec{C^*}\setminus e$. Hence $[\vec{f}] \cdot \ori O'_1 = \ori O'_2$.\footnote{Note that Theorem \ref{thm:generalDescription} implies that $\ori O'_2$ can be obtained from $\ori O'_1$ by reversing $\vec f$ and a single cocircuit. In particular, this implies that $\vec{C^*}\setminus e$ must actually be a single cocircuit. This fact is not necessary for the proof, but we feel that it merits mentioning.}\\ 

\begin{enumerate}[resume* = cases]
\item $\vec f \compat \ori O_1$ and $\vec f \compat \ori O_2$.\\ 
\end{enumerate}

For this case, we use the fact that 
\[ [\vec f] \cdot \ori O_1 = \ori O_2 \iff [-\vec f] \cdot \ori O_2 = \ori O_1.\]
This means that after switching $\ori O_1$ with $\ori O_2$ and $\vec f$ with $-\vec f$, we have reduced to Case 2.\\

\begin{enumerate}[resume* = cases]
\item $\vec f \compat \ori O_1$ and $\vec f \not\compat \ori O_2$.\\ 
\end{enumerate}

By Theorem~\ref{thm:generalDescription}, $\ori O_2$ can be obtained from $\ori O_1$ by either \begin{itemize}
\item reversing a circuit $\vec C$, reversing $f$, and then reversing a cocircuit $\vec{C^*}$, or 
\item reversing a cocircuit, reversing $f$, and then reversing a circuit. \end{itemize} 
Furthermore, using similar logic to what we used to reduce Case 3 to Case 2, we can assume that the circuit is reversed first. In particular, if the cocircuit is reversed first, we swap $\ori O_1$ with $\ori O_2$ and $\vec f$ with $-\vec f$. 

Hence, suppose that the circuit $\vec C$ is reversed first. Here, after applying Lemma~\ref{lem:newcase}, we can conclude that $e\notin C$. This means that $\ori O'_2$ can be obtained from $\ori O'_1$ by reversing $\vec{C}$, reversing $-\vec f$, and then reversing $\vec{C^*}\setminus e$. By  Lemma~\ref{lem:breaking cocircuit}, $\vec{C^*}\setminus e$ is a signed cocircuit of $M\setminus e$ or a disjoint union of signed cocircuits of $M\setminus e$. It follows that $[\vec{f}] \cdot \ori O'_1 = \ori O'_2$. 

\end{proof} 

We will show that Theorem~\ref{thm:consistency}(2) follows from Theorem~\ref{thm:consistency}(1) by a duality argument. First, we state a few straightforward results. 

\begin{lemma} \label{lem:BBY_dual_equal}
Let $B\in\mathbf{B}(M)$.
Recall that $E\setminus B\in\mathbf{B}(M^*)$ and vice versa.
We have $\BBY_{(M,\sigma,\sigma^*)}(B)=\BBY_{(M^*,\sigma^*,\sigma)}(E\setminus B)$.
\end{lemma}
\begin{proof}
Note that the fundamental cocircuit of $e\not\in B$ with respect to $E\setminus B\in\mathbf{B}(M^*)$ is the fundamental circuit of $e$ with respect to $B\in\mathbf{B}(M)$, and analogously for $e\in B$.
\end{proof}

\begin{cor}\label{cor:BBY_torsors_of_dual_matroids_are_compatible} 
The BBY torsors $\BBY_{(M,\sigma,\sigma^*)}$ and $\BBY_{(M^*,\sigma^*,\sigma)}$ agree in the sense that for any arc $\overrightarrow{f}$ and $B_1,B_2\in\mathbf{B}(M)$, if $[\overrightarrow{f}]\cdot \BBY_{(M,\sigma,\sigma^*)}(B_1)=\BBY_{(M,\sigma,\sigma^*)}(B_2)$, then $[\overrightarrow{f}] \cdot \BBY_{(M^*,\sigma^*,\sigma)}(E\setminus B_1)=\BBY_{(M^*,\sigma^*,\sigma)}(E\setminus B_2)$.
\end{cor}

\begin{proof}We can apply Lemma~\ref{lem:BBY_dual_equal} to obtain the following chain of equalities. 

    \[[\overrightarrow{f}] \cdot \BBY_{(M^*,\sigma^*,\sigma)}(E\setminus B_1)=[\overrightarrow{f}] \cdot\BBY_{(M,\sigma,\sigma^*)}(B_1)=\BBY_{(M,\sigma,\sigma^*)}(B_2)=\BBY_{(M^*,\sigma^*,\sigma)}(E\setminus B_2).\]
\end{proof}

Now, we can apply Corollary~\ref{cor:BBY_torsors_of_dual_matroids_are_compatible} to prove Theorem \ref{thm:consistency}(2).

\begin{proof}[Proof of Theorem \ref{thm:consistency}(2)]


Let $\vec f, e, B_1, B_2$ be as in the theorem statement. 
Then, by Corollary \ref{cor:BBY_torsors_of_dual_matroids_are_compatible}, we have 
\[[\overrightarrow{f}] \cdot \BBY_{(M^*,\sigma^*,\sigma)}(E\setminus B_1)=\BBY_{(M^*,\sigma^*,\sigma)}(E\setminus B_2).\]

Moreover, $e\in (B_1\cap B_2)\setminus f$ is equivalent to $e\in (E\setminus B_1)^c\cap(E\setminus B_2)^c\setminus f$, so we can apply Theorem \ref{thm:consistency}(1) for $M^*$, $\sigma^*$ and $\sigma$, and obtain that 
\[[\vec{f}] \cdot \BBY_{(M^*\setminus e,\sigma^*\setminus e,\sigma\setminus e)}(E\setminus B_1) = \BBY_{(M^* \setminus e,\sigma^*\setminus e,\sigma\setminus e)}(E\setminus B_2).\]

Hence, applying Corollary \ref{cor:BBY_torsors_of_dual_matroids_are_compatible} to $M^*\setminus e=(M/e)^*$, we get that
\[[\vec{f}] \cdot \BBY_{(M/e,\sigma/e,\sigma^*/e)}((E\setminus e)\setminus (E\setminus B_1)) = \BBY_{(M/ e,\sigma/ e,\sigma^*/ e)}((E\setminus e)\setminus (E\setminus B_2)),\]
that is,
\[[\vec{f}] \cdot \BBY_{(M/ e,\sigma/e,\sigma^*/ e)}(B_1\setminus e) = \BBY_{(M/ e,\sigma/ e,\sigma^*/ e)}(B_2\setminus e).\]
This is exactly what we wanted to prove. 
\end{proof}

Finally, we prove Theorem~\ref{thm:consistency}(3), which is relatively straightforward.

\begin{proof}[Proof of Theorem~\ref{thm:consistency}(3).]
By Theorem \ref{thm:generalDescription}, to get from $\BBY_{(M,\sigma,\sigma^*)}(B_1)$ to $\BBY_{(M,\sigma,\sigma^*)}(B_2)$, we need to reverse at most one circuit and cocircuit containing $f$, then reverse $f$, then reverse again at most one signed circuit or cocircuit containing $f$.

It is immediate from the definition of a connected component that if $e$ and $f$ are in different connected components, then $e$ is not contained by any circuit containing $f$. 
Dually, a cocircuit cannot intersect more than one connected component \cite[4.2.18]{Oxley_book}. 

Hence the orientations $\BBY_{(M,\sigma,\sigma^*)}(B_1)$ and $\BBY_{(M,\sigma,\sigma^*)}(B_2)$ agree on the connected component containing $e$. As the fundamental circuits and cocircuits of elements stay within their connected component, this means that $B_1$ and $B_2$ agree on the connected component containing $e$.
\end{proof}

\section{A Uniqueness Conjecture and other Open Questions}\label{sec:further}

In this section, we will discuss some of the remaining mysteries involving sandpile torsors, which were formally defined in Section~\ref{sec:torsors}. 

By proving Theorem~\ref{thm:consistency}, we have demonstrated that the BBY sandpile torsor algorithm is consistent (with the auxiliary data of either acyclic or triangulating circuit-cocircuit signatures). There are also 3 other consistent sandpile torsor algorithms closely related to BBY. In particular, we could swap $\fouri F(B,\sigma)$ with $-\fouri F(B,\sigma)$ and/or we could swap $\fouri F(B,\sigma^*)$ with $-\fouri F(B,\sigma^*)$. The first change is equivalent to replacing $\sigma$ with $\SignedCircuits M \setminus \sigma$ while the second change is equivalent to replacing $\sigma^*$ with $\SignedCoCircuits M \setminus \sigma^*$. We say that the four resulting algorithms have the same \emph{structure}. 

\begin{conject}\label{conj:uniqueness}
    Every consistent sandpile torsor algorithm for the class of oriented regular matroids with the auxiliary data of triangulating signatures has the same structure as the BBY algorithm. 
\end{conject}


In \cite{GM}, the analogue of Conjecture~\ref{conj:uniqueness} was proved for plane graphs. One of the techniques used in~\cite{GM} was to identify a small, well-behaved set of pairs of sandpile group elements and spanning trees, whose action already determine the sandpile action. Finding such a set of pairs for the BBY action could help in proving Conjecture \ref{conj:uniqueness}, and may be of independent interest.

\begin{defn}\label{def:genpair}
    Consider $\mathbf P \subseteq (S(M) \times \mathbf B(M))$. We say that $\mathbf P$ is a \emph{set of generating pairs} for the BBY action if the following condition is satisfied. 

    For any $(s,B) \in (S(M)\times \mathbf B(M))$, there is a finite $k$ for which there exist collections $s_1,s_2,\dots,s_k \in S(M)$ and $B_1,B_2,\dots, B_k,B_{k+1} \in \mathbf B(M)$ with the following properties: 
    
    \begin{enumerate}
        \item $\sum_{i=1}^k s_i = s$
        \item For each $i\le k$, we have $(s_i,B_i) \in \mathbf P$. 
        \item $B_1 = B$. 
        \item For each $i\le k$, we have $B_{i+1} =  \Gamma^\beta_{(M,\sigma,\sigma^*)}(s_i,B_i)$. 
    \end{enumerate}
\end{defn}

Definition~\ref{def:genpair} is related to the field of \emph{reconfiguration theory} (see~\cite{Reconfiguration} for an overview). The following proposition shows how identifying a set of generating pairs can simplify the task of identifying when two sandpile torsor actions are equivalent. 

\begin{prop}
    Let $\mathbf P$ be a set of generating pairs and $\Gamma_{(M,\sigma,\sigma^*)}$ be a sandpile torsor action on $(M,\sigma,\sigma^*)$. Suppose that for all $(s,B) \in \mathbf P$, we have \[  \Gamma^\beta_{(M,\sigma,\sigma^*)}(s,B) = \Gamma_{(M,\sigma,\sigma^*)}(s,B).\] 
    Then, $\Gamma_{(M,\sigma,\sigma^*)}$ is equivalent to the BBY action. 
\end{prop}
\begin{proof}
    Immediate from the definition. 
\end{proof}

It is not hard to show that $\mathbf P = (\{[\vec f] \mid  \vec f \text{ is an arc of }M\} \times \mathbf B(M))$ is a set of generating pairs, but the action of an arbitrary arc can be quite complicated. 

\begin{open}
    Given a regular matroid and a triangulating circuit-cocircuit signature, what is the smallest set of generating pairs? 
\end{open}
Less concretely, we are curious if there are any general constructions for a set of generating pairs that is significantly smaller than $(S(M) \times \mathbf B(M))$. For our final open question, we ask about a one particular candidate. 
\begin{open}
    Let $\mathbf P$ be the set of pairs where the BBY algorithm only swaps a single element of $E$. Is $\mathbf P$ a set of generating pairs? Is there an easy way to see when an element of $(S(M) \times \mathbf B(M))$ is in $\mathbf P$? 
\end{open}

\bibliographystyle{alpha}
\bibliography{biblio.bib}

\end{document}